\documentclass[12pt]{amsart}
\usepackage{a4wide,enumerate,color}
\allowdisplaybreaks

\let\pa\partial  
\let\na\nabla  
\let\eps\varepsilon  
\newcommand{\N}{{\mathbb N}}  
\newcommand{\R}{{\mathbb R}} 
\newcommand{\diver}{\operatorname{div}}  
\newcommand{\HH}{{\mathcal H}}

\newtheorem{theorem}{Theorem}   
\newtheorem{lemma}[theorem]{Lemma}   
\newtheorem{proposition}[theorem]{Proposition}   
\newtheorem{remark}[theorem]{Remark}

\begin{document}  

\title[Cross-diffusion population systems]{Global existence analysis of 
cross-diffusion population systems for multiple species}

\author[X. Chen]{Xiuqing Chen}
\address{School of Sciences, Beijing University of Posts and Telecommunications, 
Beijing 100876, China}
\email{buptxchen@yahoo.com}

\author[E. S. Daus]{Esther S. Daus}
\address{Institute for Analysis and Scientific Computing, Vienna University of  
	Technology, Wiedner Hauptstra\ss e 8--10, 1040 Wien, Austria}
\email{esther.daus@tuwien.ac.at} 

\author[A. J\"ungel]{Ansgar J\"ungel}
\address{Institute for Analysis and Scientific Computing, Vienna University of  
	Technology, Wiedner Hauptstra\ss e 8--10, 1040 Wien, Austria}
\email{juengel@tuwien.ac.at} 

\date{\today}

\thanks{The first author acknowledges support from the National Natural Science 
Foundation of China, grant 11471050. 
The last two authors acknowledge partial support from   
the Austrian Science Fund (FWF), grants P22108, P24304, and W1245} 

\begin{abstract}
The existence of global-in-time weak solutions to reaction-cross-diffusion systems
for an arbitrary number of competing population species is proved. The equations 
can be derived from an on-lattice random-walk model with general transition rates.
In the case of linear transition rates, it extends the two-species
population model of Shigesada, Kawasaki, and Teramoto. The equations are
considered in a bounded domain with homogeneous Neumann boundary conditions. 
The existence proof is based on a refined entropy method 
and a new approximation scheme.
Global existence follows under a detailed balance or weak cross-diffusion
condition. The detailed balance condition is related to the symmetry of the
mobility matrix, which mirrors Onsager's principle in thermodynamics.
Under detailed balance (and without reaction), the entropy is nonincreasing 
in time, but counter-examples
show that the entropy may increase initially if detailed balance does not hold.
\end{abstract}



\maketitle


\section{Introduction}

Shigesada, Kawasaki, and Teramoto
suggested in their seminal paper \cite{SKT79} a diffusive Lotka-Volterra
system for two competing species, which is able to describe the segregation
of the population and to show pattern formation when time increases. 
Starting from an on-lattice random-walk model, this system was extended to an
arbitrary number of species in \cite[Appendix]{ZaJu16}. While the existence analysis
of global weak solutions to the two-species model is well understood by now 
\cite{ChJu04,ChJu06}, only very few results for the $n$-species model under very
restrictive conditions exist (see the discussion below).
In this paper, we provide for the first time
a global existence analysis for an arbitrary number
of population species using the
entropy method of \cite{Jue15}, and we reveal an astonishing 
relation between the monotonicity of the entropy and the detailed balance
condition of an associated Markov chain.

More specifically, we consider the reaction-cross-diffusion equations
\begin{equation}\label{1.eq}
  \pa_t u_i - \diver\bigg(\sum_{j=1}^n A_{ij}(u)\na u_j\bigg) = f_i(u)
	\quad\mbox{in }\Omega,\ t>0, \quad i=1,\ldots,n,
\end{equation}
with no-flux boundary and initial conditions
\begin{equation}\label{1.bic}
  \sum_{j=1}^n A_{ij}(u)\na u_j\cdot\nu = 0\quad\mbox{on }\pa\Omega,\ t>0, \quad
	u_i(\cdot,0)=u_i^0\quad\mbox{in }\Omega. 
\end{equation}
Here, $u_i$ models the density of the $i$th species, $u=(u_1,\ldots,u_n)$,
$\Omega\subset\R^d$ ($d\ge 1$) is a bounded domain with Lipschitz 
boundary, and $\nu$ is the exterior unit normal vector to $\pa\Omega$.
The diffusion coefficients are given by
\begin{equation}\label{1.A}
  A_{ij}(u) = \delta_{ij}p_i(u) + u_i\frac{\pa p_i}{\pa u_j}(u), \quad
	p_i(u) = a_{i0} + \sum_{k=1}^n a_{ik}u_k^s, \quad i,j=1,\ldots,n,
\end{equation}
where $a_{i0}$, $a_{ij}\ge 0$ and $s>0$. The functions $p_i$ are the transition rates
of the underlying random-walk model \cite{Jue16,ZaJu16}. 
The source terms $f_i$ are of Lotka-Volterra type,
\begin{equation}\label{1.f1}
  f_i(u) = u_i\bigg(b_{i0} - \sum_{j=1}^n b_{ij}u_j\bigg), \quad i=1,\ldots,n,
\end{equation}
and we suppose that $b_{i0}$, $b_{ij}\ge 0$ (competition case). Note that
\eqref{1.eq} can be written more compactly as
$$
  \pa_t u - \diver(A(u)\na u) = f(u), \quad f(u)=(f_1(u),\ldots,f_n(u)).
$$


\subsection*{State of the art}

From a mathematical viewpoint, the analysis of \eqref{1.eq}-\eqref{1.bic} is
highly nontrivial since the diffusion matrix $A(u)$ is neither symmetric nor
generally positive definite. Although the maximum principle may be applied
to prove the nonnegativity of the densities, it is generally not possible
to show upper bounds. Moreover, there
is no general regularity theory for diffusion systems, which makes the analysis
very delicate. Equations \eqref{1.eq} can be written in the form
\begin{equation}\label{1.Delta}
  \pa_t u_i - \Delta(u_ip_i(u)) = f_i(u),
\end{equation}
which allows for the proof of an $L^{2+s}$ estimate by the duality method 
\cite{DeFe15,Pie10},
but we will not exploit this method in the paper.

The case of $n=2$ species and linear transition rates $s=1$
corresponds to the original population model
of Shigesada, Kawasaki, and Teramota \cite{SKT79},
\begin{equation}\label{1.skt}
\begin{aligned}
  \pa_t u_1 - \Delta\big(u_1(a_{10} + a_{11}u_1 + a_{12}u_2)\big) &= f_1(u), \\
  \pa_t u_2 - \Delta\big(u_2(a_{20} + a_{21}u_1 + a_{22}u_2)\big) &= f_2(u).
\end{aligned}
\end{equation}
The numbers $a_{i0}$ are the diffusion coefficients, $a_{ii}$ are the self-diffusion
coefficients, and $a_{ij}$ for $i\neq j$ are called the cross-diffusion coefficients.
This model attracted a lot of attention in the mathematical literature.
The first global existence result is due to Kim \cite{Kim84} who 
studied the equations in one space dimension, neglected
self-diffusion, and assumed equal coefficients ($a_{ij}=1$). His result was
extended to higher space dimensions in \cite{Deu87}.  
Most of the papers made restrictive structural assumptions, for instance supposing
that the diffusion matrix is triangular ($a_{21}=0$), 
since this allows for the maximum 
principle in the second equation \cite{Ama89,Le02,LNW98}. Another restriction
is to suppose that the cross-diffusion coefficients
are small, since in this situation the diffusion matrix becomes positive definite
\cite{Deu87,Yag93}.

Significant progress was made by Amann \cite{Ama89} who showed that
a priori estimates in the $W^{1,p}$ norm with $p>d$ are sufficient for
the solutions to general quasilinear parabolic systems to
exist globally in time, and he applied his result to the triangular case.
The first global existence result without any restriction on the diffusion
coefficients (except positivity) was achieved in \cite{GGJ03} in one space
dimension and in \cite{ChJu04,ChJu06} in several space dimensions. 
The results were extended to the whole space in \cite{Dre08}.
The existence of global classical solutions was proved in, e.g., \cite{Le06},
under suitable conditions on the coefficients.

Nonlinear transition rates, but still for two species, were analyzed
by Desvillettes and co-workers, assuming sublinear ($0<s<1$) \cite{DLM14} or 
superlinear rates ($s>1$) and the weak cross-diffusion condition
$((s-1)/(s+1))^2a_{12}a_{21} \le a_{11}a_{22}$ \cite{DLMT15}.
Similar results, but under a slightly stronger weak cross-diffusion hypothesis, were
proved in \cite{Jue15}.

As already mentioned, there are very few results for more than two species.
The existence of positive stationary solutions and
the stability of the constant equilibrium was investigated in \cite{AnTe14,RyAh05}.
The existence of global weak solutions in one space dimension
assuming a positive definite diffusion matrix was proved in \cite{WeFu09}, 
based on Amann's results. Using an entropy approach, the global existence of solutions
was shown in \cite{DLMT15} for three species under the condition $0<s<1/\sqrt{3}$
(which guarantees that $\det(A(u))>0$).
To our knowledge, a global existence theorem under more general conditions
seems to be not available in the literature. In this paper, we prove such
a result and relate a structural condition on the coefficients $a_{ij}$ with
Onsager's principle of thermodynamics.


\subsection*{Key ideas}

Before we state the main results, let us explain our strategy.
The idea is to find a priori estimates by employing
a Lyapunov functional approach with
\begin{equation}\label{1.ent}
  \HH[u] = \int_\Omega h(u)dx = \int_\Omega \sum_{i=1}^n \pi_i h_s(u_i)dx,
\end{equation}
where $\pi_i>0$ are some numbers and
\begin{equation}\label{1.h}
	h_s(z) = \left\{\begin{array}{ll}
	z(\log z-1) + 1 &\mbox{for } s=1, \\
	\displaystyle\frac{z^s-sz}{s-1} + 1 &\mbox{for } s\neq 1.
	\end{array}\right.
\end{equation}
Because of the connection of our method to nonequilibrium thermodynamics
\cite[Section 4.3]{Jue16}, we refer to $\HH[u]$ as an entropy and to $h(u)$ 
as an entropy density. Introducing the so-called entropy variable
$w=(w_1,\ldots,w_n)$ (called chemical potential in thermodynamics) by
$$
  w_i = \frac{\pa h}{\pa u_i}(u) = \left\{\begin{array}{ll}
	\pi_i\log u_i &\mbox{for }s=1, \\
	\displaystyle\frac{s\pi_i}{s-1}(u_i^{s-1}-1) &\mbox{for }s\neq 1,
	\end{array}\right.
$$
equations \eqref{1.eq} can be written as
\begin{equation}\label{1.B}
  \pa_t u(w) - \diver(B(w)\na w) = f(u(w)), \quad B(w) = A(u)H(u)^{-1},
\end{equation}
where $u(w):=(h')^{-1}(w)$ is the inverse transformation 
and $H(u)=h''(u)$ is the Hessian of the entropy density. 
We claim that if $f=0$ and $B(w)$
or, equivalently, $H(u)A(u)$ is positive semi-definite\footnote{We say that an 
arbitrary matrix $M\in\R^{n\times n}$ is positive (semi-) definite if 
$z^\top Mz>$ ($\ge$) 0 for all $z\in\R^n$, $z\neq 0$.}, 
$\HH[u]$ is a
Lyapunov functional along solutions to \eqref{1.eq}. Indeed, a (formal) computation
shows that
$$
  \frac{d}{dt}\HH[u] = -\int_\Omega\na w:B(w)\na w dx \le 0,
$$
which implies that $t\mapsto\HH[u(t)]$ is nonincreasing. The entropy method
provides more than just the monotonicity of $\HH[u]$. If, for instance, 
$z^\top H(u)A(u)z \ge \sum_{i=1}^n c_i u_i^{\alpha-2} z_i^2$ 
for some constants $\alpha>0$, $c_i>0$, it follows that
$$
  \frac{d}{dt}\HH[u] 
	+ \frac{4}{\alpha^2}\int_\Omega\sum_{i=1}^n c_i|\na u_i^{\alpha/2}|^2 dx \le 0,
$$
which yields gradient estimates for $u_i^{\alpha/2}$.
This strategy was employed in many papers on cross-diffusion systems; see, e.g.,
\cite{ChJu04,ChJu06,DLM14,Dre08,GGJ03,Jue15,ZaJu16}. In this paper, 
we introduce two new ideas which we explain for the case $s=1$ 
($s\neq 1$ is studied below).

It is known that the entropy \eqref{1.ent} with $\pi_i=1$ is a Lyapunov functional
for the two-species model \eqref{1.skt} with $f_1=f_2=0$. 
This property is generally not
satisfied for the corresponding $n$-species system. 
Our {\em first idea} is to introduce
the numbers $\pi=(\pi_1,\ldots,\pi_n)$ in the entropy \eqref{1.ent}. 
It turns out that \eqref{1.ent} is a Lyapunov functional and
$H(u)A(u)$ is symmetric and positive definite if 
\begin{equation}\label{1.db}
  \pi_i a_{ij} = \pi_j a_{ji} \quad\mbox{for all }i,j=1,\ldots,n.
\end{equation}
More precisely, this property is {\em equivalent} to the symmetry of $H(u)A(u)$ 
(see Proposition \ref{prop.db}). We recognize \eqref{1.db} as the detailed balance
condition for the Markov chain associated to $(a_{ij})$. The equivalence of
the symmetry and the detailed balance condition is new but not surprising. 
In fact, the latter condition means that $\pi$ is a reversible
measure, and time-reversibility of a thermodynamic system is equivalent
to the symmetry of the so-called Onsager matrix $B(w)$, so symmetry and
reversibility are related both from a mathematical and physical viewpoint. 
We detail these relations in Section \ref{sec.db}. In Section \ref{sec.sublin},
we derive a refined estimate for $H(u)A(u)$ leading to
\begin{equation}\label{1.ei}
  \frac{d}{dt}\HH[u] + 4\int_\Omega\sum_{i=1}^n \pi_i a_{i0}|\na\sqrt{u_i}|^2 dx
	+ 2\int_\Omega\sum_{i=1}^n\pi_i a_{ii}|\na u_i|^2 dx \le 0,
\end{equation}
and thus giving an $H^1$ estimate for $\sqrt{u_i}$ (if $a_{i0}>0$) and
$u_i$ (if $a_{ii}>0$). This is the key estimate for the global existence 
result. (Below we also take into account the reaction terms \eqref{1.f1}.)

One may ask whether the detailed balance condition is necessary for the
monotonicity of the entropy. It is not. We show that if self-diffusion 
dominates cross-diffusion in the sense 
\begin{equation}\label{1.eta0}
  \eta_0 := \min_{i=1,\ldots,n}\bigg(a_{ii} - \frac{s}{2(s+1)}\sum_{j=1}^n
	\big(\sqrt{a_{ij}} - \sqrt{a_{ji}}\big)^2\bigg) > 0,
\end{equation}
and detailed balance may be not satisfied,
then the estimate leading to \eqref{1.ei} still holds (with different constants), 
and global existence follows. (Throughout this paper, we set $\pi_i=1$ when
detailed balance does not hold.)
However, if conditions \eqref{1.db} or \eqref{1.eta0} are both not satisfied, 
there exist coefficients
$a_{ij}$ and initial data $u^0$ such that $t\mapsto \HH[u(t)]$ is increasing
on $[0,t_0]$ for some $t_0>0$; 
see Section \ref{sec.incr}. Numerical
experiments (not shown) indicate that after the initial increase, the entropy 
decays and, in fact, it stays bounded for all time. We conjecture that
the entropy is bounded for all time for all nonnegative coefficients and 
nonnegative initial data
and that global existence of weak solutions holds for any (positive)
coefficients $a_{ij}$.

Our results can be extended to nonlinear transition rates of type \eqref{1.A}.
One may choose more general terms $a_{ij}u_j^{s_j}$ with different exponents
$s_j$ but the results are easier to formulate if all exponents are equal.  
Coefficients with exponents $s\neq 1$ were also considered in 
\cite{DLM14,DLMT15,Jue15} 
but in the two-species case only. We generalize these results to
the multi-species case for any $n\ge 2$.
The entropy method has to be adapted since the inverse of 
$h_s'(z)=(s/(s-1))(z^{s-1}-1)$ cannot be defined on $\R$ and thus,
$u(w)=(h')^{-1}(w)$ is not defined for all $w\in\R^n$. This issue can be overcome
by regularization as in \cite{DLM14,Jue15}. In fact, we introduce
$$
  h_\eps(u) = h(u) + \eps\sum_{i=1}^n\big(u_i(\log u_i-1)+1\big).
$$
Then $h_\eps':(0,\infty)^n\to\R^n$ can be inverted and 
$(h_\eps')^{-1}:\R^n\to(0,\infty)^n$ is defined on $\R^n$.
As a consequence, $u_i=(h_\eps')^{-1}(w)_i$ is positive for any $w\in\R^n$ and
even strongly positive if $w$ varies in a compact subset of $\R^n$.

Unfortunately, the product $H_\eps(u)A(u)$, where $H_\eps(u)=h_\eps''(u)$, 
is generally not positive definite
and we need to approximate $A(u)$. In contrast to the approximations suggested
in \cite{DLM14,Jue15}, we employ a {\em non-diagonal} matrix; see \eqref{2.Aeps} 
below. More specifically, we introduce $A_\eps(u)=A(u)+\eps A^0(u)+\eps^\eta A^1(u)$ 
with non-diagonal $A^0(u)$, diagonal $A^1(u)$, and $\eta\le 1/2$ such that
$$
  z^\top H_\eps(u)A_\eps(u)z \ge z^\top H(u)A(u)z \quad\mbox{for all }z\in\R^n.
$$
The choice of the non-diagonal approximation satisfying this inequality is
nontrivial, and this construction is our {\em second idea}. 


\subsection*{Main results}

First, we show that global existence of weak solutions holds for linear
transition rates ($s=1$). In the following, we set $Q_T=\Omega\times(0,T)$.

\begin{theorem}[Global existence for linear transition rates]\label{thm.ex1}
Let $T>0$, $s=1$ and $u^0=(u_1^0,\ldots,u_n^0)$ be 
such that $u_i^0\ge 0$ for $i=1,\ldots,n$ and
$\int_\Omega h(u^0)dx<\infty$. Let either detailed balance and $a_{ii}>0$ for
$i=1,\ldots,n$; or \eqref{1.eta0} hold.
Then there exists a weak solution $u=(u_1,\ldots,u_n)$ to \eqref{1.eq}-\eqref{1.bic} 
satisfying $u_i\ge 0$ in $\Omega$, $t>0$, and
\begin{align*}
  & u_i\in L^2(0,T;H^1(\Omega)), \quad
	u_i\in L^\infty(0,T;L^1(\Omega)), \\
	& u_i\in L^{2+2/d}(Q_T), \quad
	\pa_t u_i\in L^{q'}(0,T;W^{1,q}(\Omega)'), \quad i=1,\ldots,n,
\end{align*}
where $q=2(d+1)$ and $q'=(2d+2)/(2d+1)$. 
The solution $u$ solves \eqref{1.eq} in the weak sense
\begin{equation}\label{1.weak1}
  \int_0^T\langle\pa_t u,\phi\rangle dt + \int_0^T\int_\Omega\na\phi:A(u)\na u dxdt
	= \int_0^T\int_\Omega f(u)\cdot\phi dxdt
\end{equation}
for all test functions $\phi\in L^q(0,T;W^{1,q}(\Omega))$, 
and the initial condition in \eqref{1.bic} is satisfied in the sense of
$W^{1,q}(\Omega)'$.
\end{theorem}

The theorem can be generalized to the case of vanishing self-diffusion, i.e.\
$a_{ii}=0$ if detailed balance, $a_{i0}>0$, and $b_{ii}>0$ hold; see
Remark \ref{rem.ex1}.

Our second result is concerned with nonlinear transition rates ($s\neq 1$).
The entropy inequality yields the regularity 
$u_i\in L^{2s+2/d}(Q_T)$ which may not include $L^2$
for ``small'' exponents $s<1$ and large dimensions $d$.
For this reason, we need to suppose, in the sublinear case, 
the lower bound $s>1-2/d$ and a weaker growth of the Lotka-Volterra terms:
\begin{equation}\label{1.f2}
  f_i(u) = u_i\bigg(b_{i0} - \sum_{j=1}^n b_{ij}u_j^\sigma\bigg), \quad i=1,\ldots,n,
	\quad 0\le \sigma < 2s-1+2/d.
\end{equation}

The superlinear case ($s>1$) is somehow easier than the sublinear one since
the entropy inequality gives the higher regularity $u_i\in L^p(Q_T)$ with $p>2$.
On the other hand, we need a weak cross-diffusion constraint. 
More precisely, if detailed balance holds, we require that
\begin{equation}
	\eta_1 := \min_{i=1,\ldots,n}\bigg(a_{ii} - \frac{s-1}{s+1}\sum_{j=1,\,j\neq i}^n 
	a_{ij}\bigg) > 0, \label{1.eta1}
\end{equation}
and if detailed balance does not hold, we suppose that
\begin{equation}
	\eta_2 := \min_{i=1,\ldots,n}\bigg(a_{ii} - \frac{1}{2(s+1)}\sum_{j=1,\,j\neq i}
	\big(s(a_{ij}+a_{ji}) - 2\sqrt{a_{ij}a_{ji}}\big)\bigg) > 0. \label{1.eta2}
\end{equation}
For $m\ge 2$ and $1\le q\le\infty$ we introduce the space
\begin{equation}\label{1.wmq}
  W^{m,q}_\nu(\Omega)=\{\phi\in W^{m,q}(\Omega):\na\phi\cdot\nu=0
	\mbox{ on }\pa\Omega\}.
\end{equation}

\begin{theorem}[Global existence for nonlinear transition rates]\label{thm.ex2}
Let $T>0$, $s > \max\{0,1-2/d\}$, and $u^0$ be such that $u_i^0\ge 0$ for 
$i=1,\ldots,n$ and $\int_\Omega h(u^0)dx<\infty$.
If $s<1$, we suppose that \eqref{1.f2} and either detailed balance and $a_{ii}>0$ for
$i=1,\ldots,n$; or \eqref{1.eta0} hold.
If $s>1$, we suppose that \eqref{1.f1} and 
either detailed balance and \eqref{1.eta1} or \eqref{1.eta2} hold.
Then there exist a number $2\le q<\infty$ and a weak solution $u=(u_1,\ldots,u_n)$ to 
\eqref{1.eq}-\eqref{1.bic} satisfying $u_i\ge 0$ in $\Omega$, $t>0$, and
\begin{align*}
  & u_i^{s}\in L^2(0,T;H^1(\Omega)), \quad
	u_i\in L^\infty(0,T;L^{\max{\{1,s\}}}(\Omega)), \\
	& u_i\in L^{p(s)}(Q_T), \quad
	\pa_t u_i\in L^{q'}(0,T;W^{m,q}_\nu(\Omega)'), \quad i=1,\ldots,n,
\end{align*}
where $p(s)=2s+(2/d)\max\{1,s\}$, $1/q+1/q'=1$, and $m>\max\{1,d/2\}$. 
The solution $u$ solves \eqref{1.eq} in the ``very weak'' sense
\begin{equation}\label{1.weak2}
  \int_0^T\langle\pa_t u,\phi\rangle dt - \int_0^T\int_\Omega\sum_{i=1}^n
	u_ip_i(u)\Delta\phi_i dxdt = \int_0^T\int_\Omega f(u)\cdot\phi dxdt
\end{equation}
for all $\phi=(\phi_1,\ldots,\phi_n)\in L^q(0,T;W^{m,q}_\nu(\Omega))$, 
and the initial condition holds in the sense of $W^{m,q}_\nu(\Omega)'$.
\end{theorem}

In the superlinear case, it can be shown that the solution satisfies \eqref{1.eq}
in the weak sense \eqref{1.weak1}; see Remark \ref{rem.weak}.
Moreover, for any $s> \max\{0,1-2/d\}$,
it is sufficient to consider test functions from 
$L^\beta(0,T;W^{2,\beta}_\nu(\Omega))$ with $1/\beta+1/p(s)=1$, and the initial
condition holds in the sense of $W^{2,\beta}_\nu(\Omega)'$.
We can generalize the theorem to the case of vanishing self-diffusion if
either $s>\max\{1,d/2\}$; or $0<s<1$, $d=1$, and $\sigma<s+1$ hold; 
see Remark \ref{rem.aii}.

The lower bound $s>1-2/d$ can be avoided if the regularity
$u_i\in L^{2+s}(Q_T)$ holds, which is expected to follow from the duality method 
\cite{DeFe15,Pie10}. 
Unfortunately, this method is not compatible with our
approximation scheme (see \eqref{3.tau} below). This issue can possibly be overcome 
by employing the scheme proposed in \cite{DLMT15} which is specialized to
diffusion systems like \eqref{1.Delta}. In this paper, however, we prefer to employ
scheme \eqref{3.tau}.

The paper is organized as follows. Section \ref{sec.B} is concerned with
the positive definiteness of the matrices $H(u)A(u)$ and $H_\eps(u)A_\eps(u)$.
The existence theorems are proved in Sections \ref{sec.lin} and \ref{sec.non},
respectively. In the final Section \ref{sec.aux}, 
we detail the connection between the detailed balance condition and the
symmetry of $H(u)A(u)$, prove a nonlinear Aubin-Lions compactness lemma needed
in the proof of Theorem \ref{thm.ex2},
and show that the entropy may be increasing initially for special initial data.


\section{Positive definiteness of the mobility matrix}\label{sec.B}

We derive sufficient conditions for the positive definiteness of the matrix
$H(u)A(u)$. Let $\R_+=(0,\infty)$. Recall that 
$$
  A_{ij}(u) = \delta_{ij}\bigg(a_{i0} + \sum_{k=1}^na_{ik}u_k^s\bigg)
	+ sa_{ij}u_iu_j^{s-1}, \quad H_{ij}(u) = \delta_{ij}s\pi_iu_i^{s-2}.
$$
The following result is valid for any $s>0$.

\begin{lemma}\label{lem.HA.1}
Let $s>0$. Then, for any $z\in\R^n$ and $u\in\R_+^n$, 
\begin{align}
  z^\top H(u)A(u)z &\ge s\sum_{i=1}^n \pi_i a_{i0} u_i^{s-2}z_i^2
	+ s(1-s)\sum_{i,j=1,\,i\neq j}^n\pi_i a_{ij}u_j^s u_i^{s-2} z_i^2 \nonumber \\
	&\phantom{xx}{}+ s\sum_{i=1}^n\bigg((s+1)\pi_i a_{ii} - \frac{s}{2}\sum_{j=1}^n
	\big(\sqrt{\pi_i a_{ij}} - \sqrt{\pi_j a_{ji}}\big)^2\bigg)
	u_i^{2(s-1)}z_i^2. \label{2.HA}
\end{align}
\end{lemma}

\begin{proof}
The elements of the matrix $H(u)A(u)$ equal
\begin{align*}
  (H(u)A(u))_{ij} &= \delta_{ij}s\pi_i\bigg(a_{i0}u_i^{s-2} 
	+ \sum_{k=1}^n a_{ik}u_k^s u_i^{s-2}\bigg) + s^2 a_{ij}(u_iu_j)^{s-1} \\
	&= \delta_{ij}\big(s\pi_i a_{i0}u_i^{s-2} + s(s+1)\pi_i a_{ii}u_i^{2(s-1)}\big) \\
	&\phantom{xx}{}+ \delta_{ij}s\pi_i\sum_{k=1,\,k\neq i}^n a_{ik}u_k^s u_i^{s-2}
	+ (1-\delta_{ij})s^2\pi_i a_{ij}(u_iu_j)^{s-1}.
\end{align*}
Therefore, for $z\in\R^n$,
\begin{align}
  z^\top H(u)A(u)z &= s\sum_{i=1}^n\pi_i a_{i0}u_i^{s-2}z_i^2
	+ s(s+1)\sum_{i=1}^n\pi_i a_{ii}u_i^{2(s-1)}z_i^2 \nonumber \\
	&\phantom{xx}{}+ s\sum_{i,j=1,\,i\neq j}^n\pi_i a_{ij}u_j^su_i^{s-2}z_i^2
	+ s^2\sum_{i,j=1,\,i\neq j}^n\pi_i a_{ij}(u_iu_j)^{s-1}z_iz_j \label{2.HA2} \\
	&=: I_1 + \cdots + I_4. \nonumber
\end{align}
The sum $I_1$ is the same as the first term on the right-hand side
of \eqref{2.HA}, and $I_2$ equals the first part of the last term on this
right-hand side. The remaining terms are written as
\begin{align*}
  I_3+I_4 &= s^2\sum_{i,j=1,\,i\neq j}^n\pi_i a_{ij}u_j^su_i^{s-2}z_i^2
	+ s(1-s)\sum_{i,j=1,\,i\neq j}^n\pi_i a_{ij}u_j^su_i^{s-2}z_i^2 \\
	&\phantom{xx}{}+ s^2\sum_{i,j=1,\,i\neq j}^n\pi_i a_{ij}(u_iu_j)^{s-1}z_iz_j.
\end{align*}
The second term corresponds to the second term on the
right-hand side of \eqref{2.HA}. Thus, it remains to prove that
\begin{align*}
  J &:= s^2\sum_{i,j=1,\,i\neq j}^n\pi_i a_{ij}u_j^su_i^{s-2}z_i^2
  + s^2\sum_{i,j=1,\,i\neq j}^n\pi_i a_{ij}(u_iu_j)^{s-1}z_iz_j \\
	&\ge - \frac{s^2}{2}\sum_{j=1}^n
	\big(\sqrt{\pi_i a_{ij}} - \sqrt{\pi_j a_{ji}}\big)^2 u_i^{2(s-1)}z_i^2.
\end{align*}
For this, we employ twice the inequality $b^2+c^2\ge 2bc$:
\begin{align*}
  J &= s^2\sum_{i,j=1,\,i<j}^n\pi_i a_{ij}u_j^su_i^{s-2}z_i^2
	+ s^2\sum_{i,j=1,\,i>j}^n\pi_i a_{ij}u_j^su_i^{s-2}z_i^2 \\
  &\phantom{xx}{}+ s^2\sum_{i,j=1,\,i<j}^n\pi_i a_{ij}(u_iu_j)^{s-1}z_iz_j
	+ s^2\sum_{i,j=1,\,i>j}^n\pi_i a_{ij}(u_iu_j)^{s-1}z_iz_j \\
	&= s^2\sum_{i,j=1,\,i<j}^n\Big(\pi_i a_{ij}u_j^su_i^{s-2}z_i^2
	+ \pi_ja_{ji}u_i^su_j^{s-2}z_j^2 
	+ (\pi_ia_{ij}+\pi_ja_{ji})(u_iu_j)^{s-1}z_iz_j\Big) \\
	&\ge s^2\sum_{i,j=1,\,i<j}^n\Big(2\sqrt{\pi_ia_{ij}\pi_ja_{ji}}(u_iu_j)^{s-1}
	|z_iz_j| - (\pi_ia_{ij}+\pi_ja_{ji})(u_iu_j)^{s-1}|z_iz_j|\Big) \\
	&= -s^2\sum_{i,j=1,\,i<j}^n\big(\sqrt{\pi_ia_{ij}}-\sqrt{\pi_ja_{ji}}\big)^2
	\big|(u_i^{s-1}z_i)(u_j^{s-1}z_j)\big| \\
	&\ge -\frac{s^2}{2}\sum_{i,j=1,\,i<j}^n
	\big(\sqrt{\pi_ia_{ij}}-\sqrt{\pi_ja_{ji}}\big)^2
	\big((u_i^{s-1}z_i)^2+(u_j^{s-1}z_j)^2\big) \\
	&= -\frac{s^2}{2}\sum_{i,j=1,\,i\neq j}^n
	\big(\sqrt{\pi_ia_{ij}}-\sqrt{\pi_ja_{ji}}\big)^2(u_i^{s-1}z_i)^2.
\end{align*}
This finishes the proof.
\end{proof}


\subsection{Sublinear and linear transition rates}\label{sec.sublin}

For $s\le 1$, Lemma \ref{lem.HA.1} provides immediately the positive definiteness of
$H(u)A(u)$ if detailed balance \eqref{1.db} holds. However, we can derive
a sharper result.

\begin{lemma}[Detailed balance]\label{lem.db}
Let $0<s\le 1$ and $\pi_i a_{ij}=\pi_j a_{ji}$ for all $i\neq j$. Then,
for all $z\in\R^n$ and $u\in\R_+^n$,
\begin{align}
  z^\top H(u)A(u)z &\ge s\sum_{i=1}^n\pi_i u_i^{s-2}\big(a_{i0}+(s+1)a_{ii}u_i^s\big)
	z_i^2 \nonumber \\
	&\phantom{xx}{}+ \frac{s^2}{2}
	\sum_{i,j=1,\,i\neq j}^n \pi_i a_{ij}(u_iu_j)^{s-1}\bigg(\sqrt{\frac{u_j}{u_i}}z_i
	+ \sqrt{\frac{u_i}{u_j}}z_j\bigg)^2. \label{2.HAest}
\end{align}
\end{lemma}

\begin{proof}
The sum of the terms $I_1$ and $I_2$ in \eqref{2.HA2} is exactly the 
first term on the right-hand side of \eqref{2.HAest}. Using detailed balance, we
find that
\begin{align*}
  I_3 + I_4  
	&= \frac{s}{2}\sum_{i,j=1,\,i\neq j}^n 
	\pi_i a_{ij}(u_iu_j)^{s-1}\frac{u_j}{u_i}z_i^2
  + \frac{s}{2}\sum_{i,j=1,\,i\neq j}^n 
	\pi_i a_{ij}(u_ju_i)^{s-1}\frac{u_i}{u_j}z_j^2 \\
	&\phantom{xx}{}+ s^2\sum_{i,j=1,\,i\neq j}^n \pi_i a_{ij}(u_iu_j)^{s-1}z_iz_j \\
	&= \frac{s^2}{2}\sum_{i,j=1,\,i\neq j}^n 
	\pi_i a_{ij}(u_iu_j)^{s-1}\frac{u_j}{u_i}z_i^2
  + \frac{s^2}{2}\sum_{i,j=1,\,i\neq j}^n 
	\pi_i a_{ij}(u_ju_i)^{s-1}\frac{u_i}{u_j}z_j^2 \\
	&\phantom{xx}{}+ s^2\sum_{i,j=1,\,i\neq j}^n \pi_i a_{ij}(u_iu_j)^{s-1}z_iz_j 
	+ \frac{s}{2}(1-s)\sum_{i,j=1,\,i\neq j}^n 
	\pi_i a_{ij}(u_iu_j)^{s-1}\frac{u_j}{u_i}z_i^2 \\
	&\phantom{xx}{}+ \frac{s}{2}(1-s)\sum_{i,j=1,\,i\neq j}^n 
	\pi_i a_{ij}(u_ju_i)^{s-1}\frac{u_i}{u_j}z_j^2.
\end{align*}
The sum of the first three terms equal the second term on the right-hand side
of \eqref{2.HAest}, and the remaining two terms are nonnegative since $s\le 1$.
\end{proof}

\begin{remark}\rm\label{rem.uiuj}
In the existence proof, we will choose $z_i=\na u_i$ (with a slight abuse of 
notation). Then the first term in \eqref{2.HAest} gives 
an estimate for $\na u_i^{s/2}$ in $L^2$ (if $a_{i0}>0$) and the better bound 
$\na u_i^s\in L^2$ (if $a_{ii}>0$). If $a_{ii}=0$, we lose the latter regularity.
This loss can be compensated by the last term in \eqref{2.HAest} giving
$$
  (u_iu_j)^{s-1}\bigg|\sqrt{\frac{u_j}{u_i}}\na u_i 
	+ \sqrt{\frac{u_i}{u_j}}\na u_j	\bigg|^2
	= \frac{4}{s^2}|\na(u_iu_j)^{s/2}|^2, \quad i\neq j,
$$
and consequently a bound for $\na(u_iu_j)^{s/2}$ in $L^2$.
This observation is used in Remark \ref{rem.ex1}.
\qed
\end{remark}

\begin{lemma}[Non detailed balance]\label{lem.ndb1}
Let $0<s\le 1$. If
\begin{equation*}
  \eta_0 := \min_{i=1,\ldots,n}\bigg(a_{ii} - \frac{s}{2(s+1)}\sum_{j=1}^n
	\big(\sqrt{a_{ij}} - \sqrt{a_{ji}}\big)^2\bigg) \ge 0,
\end{equation*}
then $H(u)A(u)$ is positive definite. Under the slightly stronger condition
$\eta_0>0$, it holds for all $z\in\R^n$ and $u\in\R_+^n$ that
$$
  z^\top H(u)A(u)z \ge s\sum_{i=1}^n a_{i0} u_i^{s-2}z_i^2
	+ \eta_0 s(s+1)\sum_{i=1}^n u_i^{2(s-1)}z_i^2.
$$
\end{lemma}

The lemma follows from Lemma \ref{lem.HA.1} after choosing $\pi_i=1$ for
$i=1,\ldots,n$. Observe that $\eta_0>0$ holds if $a_{ii}>0$ for all $i$ 
and $(a_{ij})$ is symmetric.

It is possible to show the positive definiteness of $H(u)A(u)$
without any restriction on $(a_{ij})$ (except positivity) if we restrict the
choice of the parameter $s$; see the following lemma.

\begin{lemma}
Let $a_{ij}+a_{ji}>0$ for $i,j=1,\ldots,n$ and $0<s\le s_0$, where
$$
  s_0 := \min_{i,j=1,\ldots,n}\frac{2\sqrt{a_{ij}a_{ji}}}{a_{ij}+a_{ji}}\le 1.
$$
Then, for all $z\in\R^n$ and $u\in\R_+^n$,
$$
  z^\top H(u)A(u)z \ge s\sum_{i=1}^n a_{i0} u_i^{s-2}z_i^2
	+ s(s+1)\sum_{i=1}^n a_{ii}u_i^{2(s-1)}z_i^2.
$$
\end{lemma}

\begin{proof}
We choose $\pi_i=1$ for $i=1,\ldots,n$.
With the notation of the proof of Lemma \ref{lem.HA.1}, we only need to show
that $I_3+I_4\ge 0$. Employing the inequality $b^2+c^2\ge 2bc$, we find that
\begin{align*}
  I_3+I_4 &= s\sum_{i,j=1,\,i<j}^n\Big(a_{ij}u_j^s u_i^{s-2}z_i^2
	+ a_{ji}u_i^s u_j^{s-2}z_j^2 + s(a_{ij}+a_{ji})(u_iu_j)^{s-1}z_iz_j\Big) \\
	&\ge s\sum_{i,j=1,\,i<j}^n\Big(2\sqrt{a_{ij}a_{ji}}(u_iu_j)^{s-1}|z_iz_j|
	- s(a_{ij}+a_{ji})(u_iu_j)^{s-1}|z_iz_j|\Big) \\
	&= s\sum_{i,j=1,\,i<j}^n(a_{ij}+a_{ji})\bigg(
	\frac{2\sqrt{a_{ij}a_{ji}}}{a_{ij}+a_{ji}} - s\bigg)(u_iu_j)^{s-1}|z_iz_j|,
\end{align*}
and this expression is nonnegative if $s\le s_0$.
\end{proof}


\subsection{Superlinear transition rates}

Again, we assume first that detailed balance holds.

\begin{lemma}[Detailed balance]
Let $s>1$ and $\pi_i a_{ij}=\pi_j a_{ji}$ for all $i\neq j$.
If
$$
  \eta_1 :=  \min_{i=1,\ldots,n}\bigg(a_{ii} - \frac{s-1}{s+1}\sum_{j=1,\,j\neq i}^n 
	a_{ij}\bigg) \ge 0,
$$ 
then $H(u)A(u)$ is positive definite. Furthermore, if $\eta_1>0$,
then, for all $z\in\R^n$ and $u\in\R_+^n$,
$$
  z^\top H(u)A(u)z \ge s\sum_{i=1}^n\pi_i a_{i0}u_i^{s-2}z_i^2
	+ \eta_1 s(s+1)\sum_{i=1}^n\pi_i u_i^{2(s-1)}z_i^2.
$$
\end{lemma}

\begin{proof}
It is sufficient to estimate the sum $I_3+I_4$, 
defined in the proof of Lemma \ref{lem.HA.1}:
\begin{align*}
  I_3+I_4 &= s\sum_{i,j=1,\,i<j}^n\Big(\pi_i a_{ij}u_j^s u_i^{s-2}z_i^2
	+ \pi_j a_{ji}u_i^s u_j^{s-2}z_j^2 + s(\pi_i a_{ij}+\pi_j a_{ji})
	(u_iu_j)^{s-1}z_iz_j\Big) \\
	&\ge s\sum_{i,j=1,\,i<j}^n\Big(2\sqrt{\pi_ia_{ij}\pi_ja_{ji}}
	(u_iu_j)^{s-1}|z_iz_j| - s(\pi_i a_{ij}+\pi_j a_{ji})
	(u_iu_j)^{s-1}|z_iz_j|\Big) \\ 
	&= -s\sum_{i,j=1,\,i<j}^n\Big(s(\pi_i a_{ij}+\pi_j a_{ji}) 
	- 2\sqrt{\pi_ia_{ij}\pi_ja_{ji}}\Big)(u_iu_j)^{s-1}|z_iz_j| \\
	&\ge -\frac{s}{2}\sum_{i,j=1,\,i<j}^n\Big(s(\pi_i a_{ij}+\pi_j a_{ji}) 
	- 2\sqrt{\pi_ia_{ij}\pi_ja_{ji}}\Big)\big((u_i^{s-1}z_i)^2 + (u_j^{s-1}z_j)^2\big) \\
	&= -\frac{s}{2}\sum_{i,j=1,\,i\neq j}^n\Big(s(\pi_i a_{ij}+\pi_j a_{ji}) 
	- 2\sqrt{\pi_ia_{ij}\pi_ja_{ji}}\Big)(u_i^{s-1}z_i)^2.
\end{align*}
This expression simplifies because of the detailed balance condition:
$$
  I_3+I_4 \ge -s(s-1)\sum_{i,j=1,\,i\neq j}^n\pi_i a_{ij}(u_i^{s-1}z_i)^2,
$$
and we end up with
$$
  z^\top H(u)A(u)z \ge s\sum_{i=1}^n\pi_i a_{i0}u_i^{s-2}z_i^2
	+ s(s+1)\sum_{i=1}^n\pi_i\bigg(a_{ii} - \frac{s-1}{s+1}\sum_{j=1,\,j\neq i}^n a_{ij}
	\bigg)u_i^{2(s-1)}z_i^2,
$$
from which we conclude the result.
\end{proof}

\begin{remark}\rm
Let $n=2$. Then the condition $\eta_1\ge 0$ on the coefficients $(a_{ij})$ becomes
$a_{11} \ge a_{12}(s-1)/(s+1)$ and $a_{22} \ge a_{21}(s-1)/(s+1)$.
The product
$$
  a_{11}a_{22} \ge \bigg(\frac{s-1}{s+1}\bigg)^2a_{12}a_{21}
$$
is the same as the condition imposed in \cite[Section 5.1]{DLMT15} but weaker than 
$$
  a_{11}a_{22} \ge \bigg(\frac{s-1}{s}\bigg)^2 a_{12}a_{21},
$$
which was needed in \cite[Lemma 11]{Jue15}. Furthermore, under the
slightly stronger condition $\eta_1>0$, that is
$$
  a_{11}a_{22} > \bigg(\frac{s-1}{s+1}\bigg)^2a_{12}a_{21},
$$
our weak solution satisfies the stronger estimate $u_i^s\in L^2(0,T;H^1(\Omega))$
than that in \cite[Section 5.1]{DLMT15}.
\qed
\end{remark}

\begin{lemma}[Non detailed balance]
Let $s>1$ and let
$$
  \eta_2 := \min_{i=1,\ldots,n}\bigg(a_{ii} - \frac{1}{2(s+1)}\sum_{j=1,\,j\neq i}
	\big(s(a_{ij}+a_{ji}) - 2\sqrt{a_{ij}a_{ji}}\big)\bigg)\ge 0.
$$
Then $H(u)A(u)$ is positive definite. Moreover, if $\eta_2>0$, then,
for all $z\in\R^n$ and $u\in\R_+^n$,
$$
  z^\top H(u)A(u)z \ge s\sum_{i=1}^n a_{i0}u_i^{s-2}z_i^2
	+ \eta_2 s(s+1)\sum_{i=1}^n u_i^{2(s-1)}z_i^2.
$$
\end{lemma}

\begin{proof}
We choose $\pi_i=1$ for $i=1,\ldots,n$. Then, as in the previous proof,
$$
  I_3+I_4 \ge -\frac{s}{2}\sum_{i,j=1,\,i\neq j}^n \big(s(a_{ij}+a_{ji})
	- 2\sqrt{a_{ij}a_{ji}}\big)u_i^{2(s-1)}z_i^2
$$
and
\begin{align*}
  z^\top & H(u)A(u)z \ge s\sum_{i=1}^n a_{i0}u_i^{s-2}z_i^2
	+ s(s+1)\sum_{i=1}^n a_{ii}u_i^{2(s-1)}z_i^2 \\
	&\phantom{xx}{}-\frac{s}{2}\sum_{i,j=1,\,i\neq j}^n \big(s(a_{ij}+a_{ji})
	- 2\sqrt{a_{ij}a_{ji}}\big)u_i^{2(s-1)}z_i^2 \\
	&= s\sum_{i=1}^n a_{i0}u_i^{s-2}z_i^2 \\
	&\phantom{xx}{}
	+ s(s+1)\sum_{i=1}^n\bigg(a_{ii} - \frac{1}{2(s+1)}\sum_{i,j=1,\,i\neq j}^n 
	\big(s(a_{ij}+a_{ji}) - 2\sqrt{a_{ij}a_{ji}}\big)\bigg)u_i^{2(s-1)}z_i^2.
\end{align*}
By definition of $\eta_2$, the result follows.
\end{proof}


\subsection{Approximate matrices}\label{sec.matrix}

Our theory requires that the range of the derivative $h'$ equals $\R^n$.
Since this is not the case if $s\neq 1$, we need to approximate 
the entropy density and consequently also the diffusion matrix. 
The approximate entropy density
\begin{equation}\label{2.heps}
  h_\eps(u) = h(u) + \eps\sum_{i=1}^n\big(u_i(\log u_i-1)+1\big)
\end{equation}
possesses the property that the range of its derivative is $\R^n$.
We set $H(u)=h''(u)=(\delta_{ij}s\pi_i u_i^{s-2})_{i,j=1,\ldots,n}$ for its
Hessian and
\begin{align}
  H_\eps(u) &= H(u) + \eps H^0(u), \quad 
	H^0_{ij}(u) = \delta_{ij}u_i^{-1}, \nonumber \\
	A_\eps(u) &= A(u) + \eps A^0(u) + \eps^\eta A^1(u), \label{2.Aeps}
\end{align}
where $\eta<1/2$ and
\begin{align*}
	& A^0_{ij}(u) = \delta_{ij}\frac{u_i}{\pi_i}\mu_i 
	- (1-\delta_{ij})\frac{u_i}{\pi_i}a_{ji}, \quad A^1_{ij}(u) = \delta_{ij}u_i, \\
  & \mu_i := \frac{\pi_i}{2}\sum_{j=1,\,j\neq i}^n\bigg(\frac{a_{ji}}{\pi_i}
	+ \frac{a_{ij}}{\pi_j}\bigg), \quad i=1,\ldots,n.
\end{align*}
The approximation $\eps^\eta A^1(u)$ is needed to achieve bounds for
$\eps^{(\eta+1)/2}\na u_i$ in $L^2$, which are necessary for the limit
$\eps\to 0$. The off-diagonal terms in $A^0(u)$ are needed to preserve
the entropy structure in the sense that $H_\eps(u)A_\eps(u)$ is still 
positive definite. This is shown in the following lemma.

\begin{lemma}\label{lem.HepsAeps}
Let $s>0$. Then, for all $z\in\R^n$ and $u\in\R_+^n$,
$$
  z^\top H_\eps(u)A_\eps(u)z
	\ge z^\top H(u)A(u)z + \eps^\eta s\sum_{i=1}^n \pi_{i}u_i^{s-1}z_i^2
	+ \eps^{\eta+1}\sum_{i=1}^n z_i^2.
$$
\end{lemma}

\begin{proof}
We decompose the product $H_\eps(u)A_\eps(u)$ as
\begin{align*}
  H_\eps(u)A_\eps(u) &= H(u)A(u) + \eps^\eta H_\eps(u)A^1(u)
	+ \eps\big(H^0(u)A(u) + H(u)A^0(u)\big) \\
	&\phantom{xx}{}+ \eps^2 H^0(u)A^0(u).
\end{align*}
The $\eps^2$-term becomes
\begin{align*}
  (H^0(u)A^0(u))_{ij} &= \sum_{k=1}^n\delta_{ik}u_k^{-1}
	\bigg(\delta_{kj}\frac{u_k}{\pi_k}\mu_k 
	- (1-\delta_{kj})\frac{u_k}{\pi_k}a_{jk}\bigg) \\
  &= \delta_{ij}\frac{\mu_i}{\pi_i} - (1-\delta_{ij})\frac{a_{ji}}{\pi_i}.
\end{align*}
We obtain for $z\in\R^n$:
\begin{align*}
  z^\top H^0(u)A^0(u)z &= \sum_{i=1}^n\frac{\mu_i}{\pi_i}z_i^2
	- \sum_{i,j=1,\,i\neq j}^n \frac{a_{ji}}{\pi_i}z_iz_j \\
	&\ge \sum_{i=1}^n\frac{\mu_i}{\pi_i}z_i^2
	- \frac12\sum_{i,j=1,\,i\neq j}^n \frac{a_{ji}}{\pi_i}(z_i^2+z_j^2) \\
	&= \sum_{i=1}^n\frac{\mu_i}{\pi_i}z_i^2
	- \frac{1}{2}\sum_{i=1}^n\bigg(\sum_{j=1,\,j\neq i}^n\frac{a_{ji}}{\pi_i}\bigg)z_i^2
	- \frac{1}{2}\sum_{i=1}^n\bigg(\sum_{j=1,\,j\neq i}^n\frac{a_{ij}}{\pi_j}\bigg)z_i^2
	\\
	&= 0.
\end{align*}
Next, we consider the $\eps$-terms:
\begin{align*}
  (H^0(u)A(u))_{ij} &= \sum_{k=1}^n \delta_{ik}u_i^{-1}
	\bigg(\delta_{kj}\bigg(a_{k0} + \sum_{\ell=1}^n a_{k\ell}u_\ell^s 
	+ sa_{kk}u_k^s\bigg) + (1-\delta_{kj})s a_{kj}u_j^{s-1}u_k\bigg) \\
	&= \delta_{ij}\bigg(a_{i0}u_i^{-1} + \sum_{\ell=1}^n a_{i\ell}u_\ell^s u_i^{-1}
	+ sa_{ii}u_i^{s-1}\bigg) + (1-\delta_{ij})sa_{ij} u_j^{s-1}, \\
	(H(u)A^0(u))_{ij} &= \sum_{k=1}^n \delta_{ik} s\pi_i u_i^{s-2}
	\bigg(\delta_{kj}\frac{u_k}{\pi_k}\mu_k 
	- (1-\delta_{kj})\frac{u_k}{\pi_k}a_{jk}\bigg) \\
	&= \delta_{ij} su_i^{s-1}\mu_i - (1-\delta_{ij})s a_{ji}u_i^{s-1}.
\end{align*}
Summing these expressions and neglecting some positive contributions, we find that
\begin{align*}
  z^\top & \big(H^0(u)A(u) + H(u)A^0(u)\big)z
	\ge \sum_{i=1}^n (a_{i0}u_i^{-1} + sa_{ii}u_i^{s-1})z_i^2 \\
	&\phantom{xx}{}+ s\sum_{i,j=1}^n (1-\delta_{ij})a_{ij} u_j^{s-1}z_iz_j
	- s\sum_{i,j=1}^n(1-\delta_{ij})a_{ji}u_i^{s-1}z_iz_j \\
	&= \sum_{i=1}^n \big(a_{i0}u_i^{-1} + sa_{ii}u_i^{s-1})z_i^2
	\ge s\sum_{i=1}^n a_{ii}u_i^{s-1}z_i^2.
\end{align*}
Here we see how we constructed $A_{ij}^0(u)$: The off-diagonal coefficients
are chosen in such a way that the mixed terms in $z_iz_j$ cancel, and the
diagonal elements (namely $\mu_i$) are sufficiently large to obtain positive
definiteness of $H^0(u)A^0(u)$. 
Finally, we have $(H_\eps(u)A^1(u))_{ij}$ $=\delta_{ij}(s\pi_i u_i^{s-1}+\eps)$ and
$$
  z^\top H_\eps(u)A^1(u)z = \sum_{i=1}^n(s\pi_i u_i^{s-1}+\eps)z_i^2,
$$
which proves the lemma.
\end{proof}


\section{Linear transition rates: proof of Theorem \ref{thm.ex1}}\label{sec.lin}

In this section, we prove Theorem \ref{thm.ex1}. Let $T>0$, $N\in\N$,
$\tau=T/N$, $\eps>0$, and $m\in\N$ with $m>d/2$. This ensures that the embedding 
$H^m(\Omega)\hookrightarrow L^\infty(\Omega)$ is compact.
We assume that $u_i^0(x)\in[a,b]$ for $x\in\Omega$, $i=1,\ldots,n$, where
$0<a<b<\infty$. Then, clearly, $w^0=h'(u^0)\in L^\infty(\Omega;\R^n)$. 
For general $u_i^0\ge 0$, we may first consider
$u_\eps^0=(Q_\eps(u_1^0),\ldots,Q_\eps(u_n^0))$, where $0<\eps<1$ and
$Q_\eps$ is the cut-off function
$$
  Q_\eps(z) = \left\{\begin{array}{ll}
	\eps &\mbox{for }0\le z<\eps, \\
	z    &\mbox{for }\eps\le z<\eps^{-1/2}, \\
	\eps^{-1/2} &\mbox{for }z\ge\eps^{-1/2},
	\end{array}\right.
$$
and then pass to the limit $\eps\to 0$. We leave the details to the reader.

{\em Step 1: solution of an approximated problem.}
Given $w^{k-1}\in L^\infty(\Omega;\R^n)$ for $k\in\N$, 
we wish to find $w^k\in H^m(\Omega;\R^n)$ such that
\begin{align}
  \frac{1}{\tau}\int_\Omega & \big(u(w^k)-u(w^{k-1})\big)\cdot\phi dx
	+ \int_\Omega\na\phi:B(w^k)\na w^k dx \nonumber \\
	&{}+ \eps\int_\Omega\bigg(\sum_{|\alpha|=m}D^\alpha w^k\cdot D^\alpha\phi
	+ w^k\cdot\phi\bigg)dx = \int_\Omega f(u(w^k))\cdot\phi dx \label{3.tau}
\end{align}
for all $\phi\in H^m(\Omega;\R^n)$.
Here, $u(w^k)=(h')^{-1}(w^k)$, $B(w^k)=A(u(w^k))H(u(w^k))^{-1}$,
$\alpha=(\alpha_1,\ldots,\alpha_d)\in\N_0^n$ with 
$|\alpha|=\alpha_1+\ldots+\alpha_d=m$ is a multiindex, and
$D^\alpha=\pa^{|\alpha|}/(\pa x_1^{\alpha_1}\cdots\pa x_d^{\alpha_d})$
is a partial derivative of order $m$. If $k=1$, we define $w^0=h'(u^0)$. 
Equation \eqref{3.tau} is an implicit Euler
discretization of \eqref{1.eq} including an $H^{m}$ regularization term.

We recall that the entropy is given by
$$
  \HH[u] = \int_\Omega h(u)dx = \int_\Omega\sum_{i=1}^n \pi_i h_1(u_i)dx,
	\quad h_1(u_i) = u_i(\log u_i-1)+1.
$$
Then the entropy variables equal $w_i=\pa h/\pa u_i = \pi_i\log u_i$. In particular,
$h':\R_+^n\to\R^n$ is invertible on $\R^n$, i.e., Hypothesis (H1) in
\cite{Jue15} is satisfied. By Lemmas \ref{lem.db} and \ref{lem.ndb1},
$H(u)A(u)$ is positive definite, i.e., Hypothesis (H2) in \cite{Jue15} holds as well.
(At this step, we only need that $H(u)A(u)$ is positive semi-definite.)
Furthermore, $f_i$ grows at most linearly which implies that
$$
  \sum_{i=1}^n f_i(u)\pi_i\log u_i \le C_f(1+h(u)),
$$
where $C_f>0$ depends only on $(b_{ij})$ and $\pi$. This means that Hypothesis (H3) in 
\cite{Jue15} is also satisfied.
Thus, we can apply Lemma 5 in \cite{Jue15} giving a weak solution 
$w^k\in H^m(\Omega;\R^n)$ to \eqref{3.tau} satisfying the discrete entropy inequality
\begin{align}
  (1-C_f\tau) & \int_\Omega h(u(w^k))dx 
	+ \tau\int_\Omega\na w^k:B(w^k)\na w^k dx \nonumber \\
  & {}+ \eps\tau\int_\Omega\bigg(\sum_{|\alpha|=m}|D^\alpha w^k|^2
	+ |w^k|^2\bigg)dx \le \int_\Omega h(u(w^{k-1}))dx + C_f\tau\mbox{meas}(\Omega).
	\label{3.dei}
\end{align}

{\em Step 2: uniform estimates.} We set $u^k=u(w^k)$ and introduce the
piecewise in time constant functions $w^{(\tau)}(x,t)=w^k(x)$ and
$u^{(\tau)}(x,t)=u^k(x)$ for $x\in\Omega$, $t\in((k-1)\tau,k\tau]$. 
At time $t=0$, we set $w^{(\tau)}(\cdot,0)=h'(u^0)=w^0$ and $u^{(\tau)}(\cdot,0)=u^0$.
Let $u^{(\tau)}=(u_1^{(\tau)},\ldots,u_n^{(\tau)})$. We define the backward
shift operator $(\sigma_\tau u^{(\tau)})(x,t)=u(w^{k-1}(x))$ for $x\in\Omega$,
$t\in((k-1)\tau,k\tau]$. Then $u^{(\tau)}$ solves
\begin{align}
  \frac{1}{\tau}\int_0^T & \int_\Omega(u^{(\tau)}-\sigma_\tau u^{(\tau)})\cdot\phi 
	dxdt + \int_0^T\int_\Omega\na\phi:B(w^{(\tau)})\na w^{(\tau)} dxdt \nonumber \\
	&{}+ \eps\int_0^T\int_\Omega\bigg(\sum_{|\alpha|=m}D^\alpha w^{(\tau)}\cdot
	D^\alpha\phi + w^{(\tau)}\cdot\phi\bigg)dxdt 
	= \int_0^T\int_\Omega f(u^{(\tau)})\cdot\phi dxdt \label{3.tau2}
\end{align}
for piecewise constant functions $\phi:(0,T)\to H^m(\Omega;\R^n)$. By a density
argument, this equation also holds for all $\phi\in L^2(0,T;H^m(\Omega;\R^n))$
\cite[Prop.~1.36]{Rou05}.

By Lemmas \ref{lem.db} and \ref{lem.ndb1}, we have
$$
  \na w^k:B(w^k)\na w^k
	= \na u^k:H(u^k)A(u^k)\na u^k 
	\ge 2\eta_0\sum_{i=1}^n|\na u_i^k|^2,
$$
where $\eta_0=\min_{i=1,\ldots,n}\pi_i a_{ii}>0$ if detailed balance holds, and
$\eta_0>0$ is given by \eqref{1.eta0} otherwise. By the generalized Poincar\'e
inequality \cite[Chapter 2, Section 1.4]{Tem97}, it holds that
$$
  \int_\Omega\bigg(\sum_{|\alpha|=m}|D^\alpha w^k|^2
	+ |w^k|^2\bigg)dx \ge C_P\|w^k\|_{H^m(\Omega)}^2,
$$
where $C_P>0$ is the Poincar\'e constant.
Then the discrete entropy inequality \eqref{3.dei} gives
\begin{align*}
  (1-C_f\tau) & \int_\Omega h(u^k)dx 
	+ 2\eta_0\tau\int_\Omega|\na u^k|^2 dx 
  + \eps C_P\tau\|w^k\|_{H^m(\Omega)}^2 \\
	&\le \int_\Omega h(u^{k-1})dx + C_f\tau\mbox{meas}(\Omega).
\end{align*}
Summing these inequalities over $k=1,\ldots,j$, it follows that
\begin{align*}
  (1-C_f\tau) & \int_\Omega h(u^j)dx 
	+ 2\eta_0\tau\sum_{j=1}^k\int_\Omega |\na u^k|^2 dx
  + \eps C_P\tau\sum_{j=1}^k\|w^k\|_{H^m(\Omega)}^2 \\
	&\le \int_\Omega h(u^0)dx
	+ C_f\tau\sum_{k=1}^{j-1}\int_\Omega h(u^k)dx + C_f T\mbox{meas}(\Omega).
\end{align*}
By the discrete Gronwall inequality \cite{Cla87}, if $\tau<1/C_f$,
$$
   \int_\Omega h(u^j)dx + \tau\sum_{j=1}^k\int_\Omega|\na u^k|^2dx
	+ \eps\tau \sum_{j=1}^k\|w^k\|_{H^m(\Omega)}^2 \le C,
$$
where here and in the following, $C>0$ denotes a generic constant independent
of $\tau$ and $\eps$. Then,
observing that the entropy density dominates the $L^1$ norm and consequently,
$u^{(\tau)}$ is uniformly bounded in $L^\infty(0,T;L^1(\Omega;\R^n))$, we obtain 
\begin{equation}\label{3.aubin1}
  \|u^{(\tau)}\|_{L^\infty(0,T;L^1(\Omega))}
	+ \|u^{(\tau)}\|_{L^2(0,T;H^1(\Omega))}
	+ \eps^{1/2}\|w^{(\tau)}\|_{L^2(0,T;H^m(\Omega))} \le C.
\end{equation}

We wish to derive more a priori estimates. Set $Q_T=\Omega\times(0,T)$.
The Gagliardo-Nirenberg inequality with $p=2+2/d$ and
$\theta=2d(p-1)/(dp+2p)\in[0,1]$ (such that $\theta p=2$) yields for $i=1,\ldots,n$,
\begin{align}
  \|u_i^{(\tau)}\|_{L^p(Q_T)}^p
	&= \int_0^T\|u_i^{(\tau)}\|_{L^p(\Omega)}^p dt
	\le C\int_0^T\|u_i^{(\tau)}\|_{H^1(\Omega)}^{\theta p}
	\|u_i^{(\tau)}\|_{L^1(\Omega)}^{(1-\theta)p} dt \nonumber \\
	&\le C\|u_i^{(\tau)}\|_{L^\infty(0,T;L^1(\Omega))}^{(1-\theta)p}
	\|u_i^{(\tau)}\|_{L^2(0,T;H^1(\Omega))}^{\theta p} \le C. \label{3.aux}
\end{align}
In order to apply a compactness result, we need a uniform estimate for the
discrete time derivative of $u^{(\tau)}$. 
Let $q=2(d+1)$ and $\phi\in L^q(0,T;W^{m,q}(\Omega;\R^n))$. Then
$1/p+1/q+1/2=1$ and, by H\"older's inequality,
\begin{align*}
  \frac{1}{\tau}\bigg|\int_0^T\int_\Omega(u^{(\tau)}-\sigma_\tau u^{(\tau)})
	\cdot\phi dxdt\bigg|
	&\le \sum_{i,j=1}^n\|A_{ij}(u^{(\tau)})\|_{L^{p}(Q_T)}\|\na u_j^{(\tau)}\|_{L^2(Q_T)}
	\|\na\phi_i\|_{L^q(Q_T)} \\
	&\phantom{xx}{}+ \eps\|w^{(\tau)}\|_{L^2(0,T;H^m(\Omega))}
	\|\phi\|_{L^2(0,T;H^m(\Omega))} \\
	&\phantom{xx}{}+ \|f(u^{(\tau)})\|_{L^{q'}(Q_T)}\|\phi\|_{L^q(Q_T)},
\end{align*}
where $q'=(2d+2)/(2d+1)$. 
Estimate \eqref{3.aux} and the linear growth of $A_{ij}(u^{(\tau)})$ with respect
to $u^{(\tau)}$ show that the first term on the right-hand side is 
bounded. The second term is bounded because of \eqref{3.aubin1}. Finally,
$|f_i(u^{(\tau)})|$ is growing at most like $(u_i^{(\tau)})^2$ such that
$$
  \|f(u^{(\tau)})\|_{L^{q'}(Q_T)} \le C\big(1+\|u^{(\tau)}\|_{L^{2q'}(Q_T)}^2\big)
	\le C,
$$
since $2q'\le p$. We conclude that
\begin{equation}\label{3.aubin2}
  \tau^{-1}\|u^{(\tau)}-\sigma_\tau u^{(\tau)}\|_{L^{q'}(0,T;W^{m,q}(\Omega)')}
	\le C.
\end{equation}

{\em Step 3: the limit $(\eps,\tau)\to 0$.}
In view of \eqref{3.aubin1} and \eqref{3.aubin2}, we can apply the Aubin-Lions
lemma in the version of \cite{DrJu12}, which yields the existence of a subsequence,
which is not relabeled, such that, as $(\tau,\eps)\to 0$,
\begin{align}
  u^{(\tau)}\to u &\quad\mbox{strongly in }L^2(Q_T)\mbox{ and a.e.}, 
	\label{3.conv1} \\
  u^{(\tau)}\rightharpoonup u &\quad\mbox{weakly in }L^2(0,T;H^1(\Omega)), 
	\label{3.conv2} \\ 
	\eps w^{(\tau)}\to 0 &\quad\mbox{strongly in }L^2(0,T;H^m(\Omega)), 
	\label{3.conv3} \\
	\tau^{-1}(u^{(\tau)}-\sigma_\tau u^{(\tau)}) \rightharpoonup \pa_t u
	&\quad\mbox{weakly in }L^{q'}(0,T;W^{m,q}(\Omega)'), \label{3.conv4}
\end{align}
where $u=(u_1,\ldots,u_n)$. In view of the a.e.\
convergence \eqref{3.conv1} and the uniform bound \eqref{3.aux}, we have
\begin{equation}\label{3.conv5}
  u^{(\tau)}\to u\quad\mbox{strongly in }L^\gamma(Q_T)\mbox{ for all }\gamma<2+2/d.
\end{equation}
Then, together with \eqref{3.conv2}, 
$$
  u_i^{(\tau)}\na u_j^{(\tau)} \rightharpoonup u_i\na u_j
	\quad\mbox{weakly in }L^1(Q_T).
$$
We deduce from the $L^{q'}(Q_T)$ bound for $A(u^{(\tau)})\na u^{(\tau)}$ that 
$$
  B(w^{(\tau)})\na w^{(\tau)}
	= A(u^{(\tau)})\na u^{(\tau)}\rightharpoonup A(u)\na u 
	\quad\mbox{weakly in }L^{q'}(Q_T). 
$$
Furthermore, taking into account \eqref{3.conv5} and the uniform bound for
$f_i(u^{(\tau)})$ in $L^{q'}(Q_T)$, 
$$
  f_i(u^{(\tau)})\rightharpoonup f_i(u) \quad\mbox{weakly in }L^{q'}(Q_T).
$$
Then \eqref{3.conv3} and \eqref{3.conv4} allow
us to perform the limit $(\eps,\tau)\to 0$ in \eqref{3.tau2} with
$\phi\in L^q(0,T;$ $W^{m,q}(\Omega))$, which directly yields \eqref{1.weak1}. 
Since $\pa_t u=\diver(A(u)\na u)+f(u)\in L^{q'}(0,T;W^{1,q}(\Omega)')$, 
a density argument shows that the weak formulation holds for all 
$\phi\in L^q(0,T;W^{1,q}(\Omega))$. Moreover, $u_i\in W^{1,q'}(0,T;W^{1,q}(\Omega)')$
$\hookrightarrow C^0([0,T];W^{1,q}(\Omega)')$, which shows that
the initial condition is satisfied in $W^{1,q}(\Omega)'$. 
This ends the proof.

\begin{remark}[Detailed balance and vanishing self-diffusion]
\rm\label{rem.ex1} 
In the detailed balance case, we may allow for vanishing self-diffusion.
If $a_{ii}=0$ but $a_{i0}>0$,
Lemma \ref{lem.db} implies that only $\na(u_i^{(\tau)})^{1/2}$ is bounded
in $L^2(Q_T)$. This situation was considered in \cite{ChJu06} for the two-species
case, and we sketch the generalization to the $n$-species case.

Applying the Gagliardo-Nirenberg inequality similarly as in Step 2 
of the previous proof, we conclude that 
$(u_i^{(\tau)})^{1/2}\in L^{\widetilde p}(Q_T)$ with $\widetilde p=2+4/d$. Then 
$$
  \|\na u_i^{(\tau)}\|_{L^{\widetilde q}(Q_T)} 
	= 2\|(u_i^{(\tau)})^{1/2}\|_{L^{\widetilde p}(Q_T)}\
	\|\na(u_i^{(\tau)})^{1/2}\|_{L^2(Q_T)}
	\le C, \quad \widetilde q = \frac{d+2}{d+1},
$$
and thus, $(u_i^{(\tau)})$ is bounded in 
$L^{\widetilde q}(0,T;W^{1,\widetilde q}(\Omega))$ 
instead of $L^2(0,T;H^1(\Omega))$. 
This loss of regularity is problematic for the estimate of the discrete
time derivative of $u_i^{(\tau)}$. In order to compensate this, 
we need the last sum in \eqref{2.HAest}.
Indeed, Remark \ref{rem.uiuj} shows that for any $i\neq j$,
$(u_i^{(\tau)}u_j^{(\tau)})^{1/2}$ is bounded in 
$L^2(0,T;H^1(\Omega))$. Moreover, $(u_i^{(\tau)}u_j^{(\tau)})^{1/2}$ is bounded in 
$L^\infty(0,T;L^1(\Omega))$.
We infer from the Gagliardo-Nirenberg inequality that 
$(u_i^{(\tau)}u_j^{(\tau)})^{1/2}$ is bounded in $L^p(Q_T)$ with $p=2+2/d$.

Next we exploit the structure of the equations,
$$
  \sum_{j=1}^n A_{ij}(u^{(\tau)})\na u_j^{(\tau)} 
	= \na(u_i^{(\tau)} p_i(u^{(\tau)})), \quad
	p_i(u^{(\tau)}) = a_{i0} + \sum_{j=1}^n a_{ij}u_j^{(\tau)}.
$$
Thus, to show that $A_{ij}(u^{(\tau)})\na u_j^{(\tau)}$ is bounded, 
we only need to verify that $\na(u_i^{(\tau)}u_j^{(\tau)})$ is bounded:
$$
  \|\na(u_i^{(\tau)}u_j^{(\tau)})\|_{L^{q'}(Q_T)}
	\le 2\|(u_i^{(\tau)}u_j^{(\tau)})^{1/2}\|_{L^p(Q_T)}
  \|\na(u_i^{(\tau)}u_j^{(\tau)})^{1/2}\|_{L^2(Q_T)} \le C,
$$ 
where $q'=(2d+2)/(2d+1)$.
The estimate for the Lotka-Volterra term is more delicate since we have only
the regularity $u_i^{(\tau)}\in L^{1+1/d}(Q_T)$. Here, we need to suppose that 
$b_{ii}>0$, since this assumption provides an estimate for 
$(u_i^{(\tau)})^2\log u_i^{(\tau)}$ in $L^1(Q_T)$.
Then the discrete time derivative of $u_i^{(\tau)}$ is bounded in 
$L^1(0,T;W^{m,q}(\Omega)')$ -- but not in $L^{q'}(0,T;W^{m,q}(\Omega)')$.
By the Aubin-Lions lemma, there exists a subsequence (not relabeled) such that,
as $(\eps,\tau)\to 0$,
$$
  u_i^{(\tau)}\to u_i \quad\mbox{strongly in }L^{q'}(Q_T).
$$
The problem now is to show that (a subsequence of) the discrete time derivative 
of $u_i^{(\tau)}$ converges to $\pa_t u_i$ since $L^1(0,T;W^{m,q}(\Omega)')$ 
is not reflexive. 
The idea is to apply a result from \cite{Yos74} which provides a criterium for
weak compactness in $L^1(0,T;X)$, where $X$ is a reflexive Banach space.
For details, we refer to \cite{ChJu06}.
\qed
\end{remark}


\section{Nonlinear transition rates: proof of Theorem \ref{thm.ex2}}\label{sec.non}

The strategy of the proof is similar
to the proof of Theorem \ref{thm.ex1} but the nonlinear transition rates
complicate the proof significantly. 
As outlined in Section \ref{sec.matrix}, we approximate
the entropy density by \eqref{2.heps} and the diffusion matrix by \eqref{2.Aeps}.
Again, we assume without loss of generality
that $u_i^0(x)\in[a,b]$ for $x\in\Omega$, $i=1,\ldots,n$, where $0<a<b<\infty$.

{\em Step 1: solution of an approximated problem.} We employ the
transformation $w_i=\pa h_\eps/\pa u_i$ and define 
$B_\eps(w)=A_\eps(u(w))H_\eps(u(w))^{-1}$. Given $w^{k-1}\in L^\infty(\Omega;\R^n)$,
we wish to find $w^k\in H^m(\Omega;\R^n)$ solving
\begin{align}
  \frac{1}{\tau}\int_\Omega & \big(u(w^k)-u(w^{k-1})\big)\cdot\phi dx
	+ \int_\Omega\na\phi:B_\eps(w^k)\na w^k dx \nonumber \\
	&{}+ \eps\int_\Omega\bigg(\sum_{|\alpha|=m}D^\alpha w^k\cdot D^\alpha\phi
	+ w^k\cdot\phi\bigg)dx = \int_\Omega f(u(w^k))\cdot\phi dx \label{4.aux}
\end{align}
for all $\phi\in H^m(\Omega;\R^n)$. If $k=1$, we define $w^0=h_\eps'(u^0)$
such that $u(w^0) = u^0$.

The construction of $h_\eps$ ensures that Hypothesis (H1) of \cite{Jue15}
is satisfied. By Lemma \ref{lem.HepsAeps}, Hypothesis (H2) holds as well.
Also Hypothesis (H3) holds true since, for some $C_f>0$,
$$
  f(u)\cdot w 
	= \sum_{I=1}^n\bigg(b_{i0}-\sum_{j=1}^n b_{ij}u_j^\sigma\bigg)
	(su_i^s + \eps u_i\log u_i) \le C_f(1+h_\eps(u)),
$$
where $\sigma=1$ if $s>1$ and $0\le\sigma\le\max\{0,2s-1+2/d\}$ if $s<1$.
We apply Lemma 5 in \cite{Jue15} to deduce the existence of a weak solution 
$w^k\in H^m(\Omega;\R^n)$ to the above problem, which satisfies the
discrete entropy inequality
\begin{align}
  (1-C_f\tau) & \int_\Omega h_\eps(u(w^k))dx 
	+ \tau\int_\Omega\na w^k:B_\eps(w^k)\na w^k dx \nonumber \\
  & {}+ \eps\tau\int_\Omega\bigg(\sum_{|\alpha|=m}|D^\alpha w^k|^2
	+ |w^k|^2\bigg)dx \le \int_\Omega h_\eps(u(w^{k-1}))dx + C_f\tau\mbox{meas}(\Omega).
	\label{4.ei}
\end{align}
Setting $u^k:=u(w^k)$ and employing 
Lemma \ref{lem.HepsAeps}, the second integral can be estimated as follows:
\begin{align}
  \int_\Omega\na w^k:B_\eps(w^k)\na w^k dx
	&= \int_\Omega u^k:H_\eps(u^k)A_\eps(u^k)\na u^k dx \nonumber \\
	&\ge s(s+1)\int_\Omega\sum_{i=1}^n\min\{a_{ii}\pi_i,\eta_0,\eta_1\pi_i,\eta_2\}
	(u_i^k)^{2(s-1)}|\na u_i^k|^2 dx \nonumber \\
	&\phantom{xx}{}
	+ \eps^\eta s\int_\Omega\sum_{i=1}^n \pi_{i}(u_i^k)^{s-1}|\na u_i^k|^2 dx 
	+ \eps^{\eta+1}\int_\Omega\sum_{i=1}^n|\na u_i^k|^2 dx \label{4.aux2} \\
	&\ge C_s\int_\Omega\sum_{i=1}^n|\na (u_i^k)^{s}|^2 dx \nonumber \\
  &\phantom{xx}{}+ \frac{4\eps^\eta s}{(s+1)^2}\int_\Omega\sum_{i=1}^n 
	\pi_{i}|\na (u_i^k)^{(s+1)/2}|^2 dx
	+ \eps^{\eta+1}\int_\Omega\sum_{i=1}^n|\na u_i^k|^2 dx, \nonumber
\end{align}
where $C_s=s^{-1}(s+1)\min\{a_{11}\pi_1,\ldots,a_{nn}\pi_n,\eta_0,\eta_1\pi_1,
\ldots,\eta_1\pi_n,\eta_2\}$.

To finish this step, we wish to write the ``very weak'' formulation for
the solution $u^{(\tau)}$, which is defined from $u^k$ as in the previous section.
First, we observe that
\begin{align*}
  (B_\eps(w^k)\na w^k)_i 
	&= (A_\eps(u^k)\na u^k)_i 
	= \eps (A^0(u^k)\na u^k)_i + \eps^\eta (A^1(u^k)\na u^k)_i + \na(u_i^k p_i(u^k)) \\
	&= \eps (A^0(u^k)\na u^k)_i + \frac{\eps^\eta}{2}\na(u_i^k)^2 + \na(u_i^k p_i(u^k)).
\end{align*}
Next, we choose a test function 
$\phi=(\phi_1,\ldots,\phi_n)\in L^q(0,T;$ $W^{m,q}_\nu(\Omega))$, where 
$m>\max\{1,d/2\}$ and $q\ge 2$ 
will be determined below. Recall that $W^{m,q}_\nu(\Omega)$ is defined in
\eqref{1.wmq}.
Integrating by parts in \eqref{4.aux}, $u^{(\tau)}$ solves
\begin{align}
  \frac{1}{\tau} & \int_0^T\int_\Omega \big(u^{(\tau)}-\sigma_\tau u^{(\tau)}\big)
	\cdot\phi dxdt - \int_0^T\int_\Omega\sum_{i=1}^n u^{(\tau)}_ip_i(u^{(\tau)})
	\Delta\phi_i dxdt \nonumber \\
	&{}+ \eps\int_0^T\int_\Omega\na\phi:A^0(u^{(\tau)})\na u^{(\tau)}dxdt 
	- \frac{\eps^\eta}{2}\int_0^T\int_\Omega\sum_{i=1}^n (u_i^{(\tau)})^2
	\Delta\phi_i dxdt  \label{4.aux3} \\
	&{}+ \eps\int_0^T\int_\Omega\bigg(\sum_{|\alpha|=m}D^\alpha w^{(\tau)}\cdot 
	D^\alpha\phi + w^{(\tau)}\cdot\phi\bigg)dx dt
	= \int_0^T\int_\Omega f(u^{(\tau)})\cdot\phi dxdt. \nonumber
\end{align}

{\em Step 2: uniform estimates.} 
Arguing as in Step 2 of the proof of Theorem \ref{thm.ex1},
we obtain from \eqref{4.ei} and \eqref{4.aux2} for suffiently small $\tau>0$
the following uniform estimates.

\begin{lemma}\label{lem.est}
It holds for $i=1,\ldots,n$ that
\begin{align}
  \|u_i^{(\tau)}\|_{L^\infty(0,T;L^{\max\{1,s\}}(\Omega))}
	+ \|(u_i^{(\tau)})^s\|_{L^2(0,T;H^1(\Omega))} &\le C, \label{4.aubin1} \\
	\eps^{\eta/2}\|(u_i^{(\tau)})^{(s+1)/2}\|_{L^2(0,T;H^1(\Omega))}
	+ \eps^{(\eta+1)/2}\|u_i^{(\tau)}\|_{L^2(0,T;H^1(\Omega))} &\le C, \label{4.eta} \\
	\eps^{1/2}\|w_i^{(\tau)}\|_{L^2(0,T;H^m(\Omega))} &\le C. \label{4.eps}
\end{align}
\end{lemma}

Here, we used the fact that $\int_\Omega h_\eps(u^0)dx$ is uniformly bounded
and that $s<1$ implies that $u^{(\tau)}\le C(1+h(u^{(\tau)}))$ for some $C>0$, 
from which we deduce that $(u_i^{(\tau)})$ is bounded 
in $L^\infty(0,T;L^1(\Omega))$. We need more a priori estimates. 

\begin{lemma}\label{lem.Lp}
Let $s>\max\{0,1-2/d\}$. It holds that
\begin{equation}
  \|u^{(\tau)}\|_{L^{p(s)}(Q_T)} + \eps^{\eta/r(s)}\|u^{(\tau)}\|_{L^{r(s)}(Q_T)}
	\le C, \label{4.Lp}
\end{equation}
where $p(s) = 2s+(2/d)\max\{1,s\}$ and $r(s)=s+1+(2/d)\max\{1,s\}>2$.
\end{lemma}

\begin{proof}
The estimates are consequences of Lemma \ref{lem.est} and the Gagliardo-Nirenberg
inequality. First, let $s<1$. We employ the Gagliardo-Nirenberg inequality, 
with $\theta=ds/(ds+1)\in(0,1)$:
\begin{align*}
  \|u_i^{(\tau)}\|_{L^{p(s)}(Q_T)}^{p(s)}
	&= \int_0^T\|(u_i^{(\tau)})^s\|_{L^{p(s)/s}(\Omega)}^{p(s)/s} dt
	\le C\int_0^T\|(u_i^{(\tau)})^s\|_{H^1(\Omega)}^{\theta p(s)/s}
	\|(u_i^{(\tau)})^s\|_{L^{1/s}(\Omega)}^{(1-\theta)p(s)/s}dt \\
	&\le C\|u_i^{(\tau)}\|_{L^\infty(0,T;L^{1}(\Omega))}^{(1-\theta)p(s)}
	\int_0^T\|(u_i^{(\tau)})^s\|_{H^1(\Omega)}^{\theta p(s)/s} dt, \quad i=1,\ldots,n.
\end{align*}
It holds that $\theta p(s)/s=2$. By
\eqref{4.aubin1}, $\|u^{(\tau)}\|_{L^{p(s)}(Q_T)}\le C$.

Next, let $s>1$. Then, with $\theta=d/(d+1)\in(0,1)$,
\begin{align*}
  \|(u_i^{(\tau)})^s\|_{L^{2+2/d}(Q_T)}^{2+2/d}
	&\le C\int_0^T\|(u_i^{(\tau)})^s\|_{H^1(\Omega)}^{\theta(2+2/d)}
	\|(u_i^{(\tau)})^s\|_{L^1(\Omega)}^{(1-\theta)(2+2/d)}dt \\
	&\le C\|(u_i^{(\tau)})^s\|_{L^2(0,T;H^1(\Omega)}^2
	\|u_i^{(\tau)}\|_{L^\infty(0,T;L^s(\Omega))}^{s(1-\theta)(2+2/d)} \le C,
\end{align*}
again taking into account estimate \eqref{4.aubin1}. 
This shows that $(u^{(\tau)})$ is bounded in $L^{p(s)}(Q_T)$.

Finally, let $\max\{0,1-2/d\}<s<1$. Then $r(s)=s+1+2/d$.
We apply the Gagliardo-Nirenberg inequality
with $\theta=d(s+1)/(2+d(s+1))\in(0,1)$ such that $\theta\cdot 2r(s)/(s+1)=2$,
\begin{align*}
  \eps^\eta\|u_i^{(\tau)}\|_{L^{r(s)}(Q_T)}^{r(s)} 
  &= \eps^{\eta}\|(u_i^{(\tau)})^{(s+1)/2}\|_{L^{2r(s)/(s+1)}(Q_T)}^{2r(s)/(s+1)} \\
	&\le \eps^{\eta}C\int_0^T\|(u_i^{(\tau)})^{(s+1)/2}
	\|_{H^1(\Omega)}^{2r(s)\theta/(s+1)}
	\|(u_i^{(\tau)})^{(s+1)/2}\|_{L^{2/(s+1)}(\Omega)}^{2r(s)(1-\theta)/(s+1)}dt \\
	&\le C\eps^{\eta}\|(u_i^{(\tau)})^{(s+1)/2}\|_{L^2(0,T;H^1(\Omega))}^2
	\|u_i^{(\tau)}\|_{L^\infty(0,T;L^1(\Omega))}^{(1-\theta)r(s)}
	\le C,
\end{align*}
using \eqref{4.aubin1} and \eqref{4.eta}. If $s>1$, we have $r(s)=s+1+2s/d$,
and applying the Gagliardo-Nirenberg inequality with 
$\theta=d(s+1)/(2s+d(s+1))\in(0,1)$, we obtain in a similar way as above
\begin{align*}
  \eps^\eta\|u_i^{(\tau)}\|_{L^{r(s)}(Q_T)}^{r(s)} 
  &= \eps^{\eta}\|(u_i^{(\tau)})^{(s+1)/2}\|_{L^{2r(s)/(s+1)}(Q_T)}^{2r(s)/(s+1)} \\
	&\le \eps^{\eta}C\int_0^T\|(u_i^{(\tau)})^{(s+1)/2}
	\|_{H^1(\Omega)}^{2r(s)\theta/(s+1)}
	\|(u_i^{(\tau)})^{(s+1)/2}\|_{L^{2s/(s+1)}(\Omega)}^{2r(s)(1-\theta)/(s+1)}dt \\
	&\le C\eps^{\eta}\|(u_i^{(\tau)})^{(s+1)/2}\|_{L^2(0,T;H^1(\Omega))}^2
	\|u_i^{(\tau)}\|_{L^\infty(0,T;L^s(\Omega))}^{(1-\theta)r(s)}
	\le C.
\end{align*}
This shows the lemma.
\end{proof}

\begin{lemma}\label{lem.ut}
Let $s>\max\{0,1-2/d\}$ and $m>\max\{1,d/2\}$. 
Then there exist $2\le q<\infty$ and $C>0$ such that
\begin{equation}\label{4.aubin2}
  \tau^{-1}\|u^{(\tau)}-\sigma_\tau u^{(\tau)}\|_{L^{q'}(0,T;W^{m,q}(\Omega)')}
	\le C,
\end{equation}
and $1/q+1/q'=1$.
\end{lemma}

\begin{proof}
Let $\phi\in L^q(0,T;W^{m,q}_\nu(\Omega))$, where $q\ge 2$ has to be determined. 
Recall that $W^{m,q}_\nu(\Omega)$ is defined in \eqref{1.wmq} and that
$m>\max\{1,d/2\}$. Then, by \eqref{4.aux3},
\begin{align}
  \tau^{-1} & \bigg|\int_\Omega(u^{(\tau)}-\sigma_\tau u^{(\tau)})\cdot\phi dx\bigg|
	\le \eps\|w^{(\tau)}\|_{L^2(0,T;H^m(\Omega))}\|\phi\|_{L^2(0,T;H^m(\Omega))} 
	\nonumber \\
	&{}+ \sum_{i=1}^n\|u_i^{(\tau)}p_i(u^{(\tau)})\|_{L^{q'}(Q_T)}
	\|\Delta\phi_i\|_{L^q(Q_T)} 
	+ \eps\sum_{i,j=1}^n\|A_{ij}^0(u^{(\tau)})\na u_j^{(\tau)}\|_{L^{q'}(Q_T)}
	\|\na\phi_j\|_{L^q(Q_T)} \nonumber \\
	&{}+ \frac{\eps^\eta}{2}\sum_{i=1}^n\|(u_i^{(\tau)})^2\|_{L^{q'}(Q_T)}
	\|\Delta\phi_i\|_{L^q(Q_T)}
	+ \|f(u^{(\tau)})\|_{L^{q'}(Q_T)}\|\phi\|_{L^q(Q_T)} \label{4.time} \\
	&=: I_1+\cdots+I_5, \nonumber
\end{align}
where $1/q+1/q'=1$. 

By \eqref{4.eps}, $I_1$ is bounded.
We deduce from \eqref{4.Lp} that $u_i^{(\tau)}(u_j^{(\tau)})^s$ is uniformly
bounded in $L^{p(s)/(s+1)}(Q_T)$, and so does $u_i^{(\tau)} p_i(u^{(\tau)})$. 
As $s>1-2/d$, we have $q_1:=p(s)/(s+1)>1$. We conclude that $I_2$ is bounded
with $q'\le \min\{2,q_1\}$.

Since $A_{ij}^0(u^{(\tau)})$ depends linearly on $u^{(\tau)}$, it is
sufficient to prove that $\eps u_i^{(\tau)}\na u_j^{(\tau)}$ is uniformly
bounded in some $L^{q_2}(Q_T)$ for all $i,j$. 
Let $q_2=2r(s)/(r(s)+2)$, where $r(s)=s+1+2/d$ is defined in Lemma \ref{lem.Lp}.
As $r(s)>2$, it holds that $q_2>1$.
Then, by H\"older's inequality, \eqref{4.eta}, and \eqref{4.Lp},
\begin{align*}
  \eps^{\eta/r(s)+(\eta+1)/2}\|u_i^{(\tau)}\na u_j^{(\tau)}\|_{L^{q_2}(Q_T)}
	&\le \eps^{\eta/r(s)}\|u_i^{(\tau)}\|_{L^{r(s)}(Q_T)}
	\cdot\eps^{(\eta+1)/2}\|\na u_j^{(\tau)}\|_{L^2(Q_T)} \le C.
\end{align*}
The property $r(s)>2$ also implies that $\eta/r(s)+(\eta+1)/2<1$. 
This shows the bound on $I_3$ with $q'\le \min\{2,q_2\}$.

Set $q_3=r(s)/2>1$. Using the second estimate in \eqref{4.Lp} and $1-2/r(s)>0$, 
we find that
$$
  \eps^\eta\|(u_i^{(\tau)})^2\|_{L^{q_3}(Q_T)} 
	= \eps^{(1-2/r(s))\eta}\big(\eps^{\eta/r(s)}\|u_i^{(\tau)}\|_{L^{r(s)}(Q_T)}\big)^2
	\le C,
$$
proving that $I_4$ is bounded with $q'\le \min\{2,q_3\}$. 

Finally, in view of \eqref{1.f2}, 
$|f_i(u^{(\tau)})|$ grows at most like $(u_i^{(\tau)})^{1+\sigma}$, 
where $\sigma=1$ if $s>1$ and $\sigma < 2s-1+2/d$ if $s<1$.
Therefore, we have $q_4:=p(s)/(1+\sigma)>1$ and
$$
  \|f(u^{(\tau)})\|_{L^{q_4}(Q_T)} 
	\le C\big(1+\|u^{(\tau)}\|_{L^{(1+\sigma)q_4}(Q_T)}^{1+\sigma}\big) 
	= C\big(1+\|u^{(\tau)}\|_{L^{p(s)}(Q_T)}^{1+\sigma}\big) \le C.
$$ 
Hence, $I_5$ is bounded with $q'\le \min\{2,q_4\}$. 
We conclude that the lemma follows with
$q':=\min\{2,q_1,q_2,$ $q_3,q_4\}>1$ and $q=q'/(q'-1)$.
\end{proof}

{\em Step 3: the limit $(\eps,\tau)\to 0$.}
Estimates \eqref{4.aubin1} and \eqref{4.aubin2} allow us to apply the nonlinear
Aubin-Lions lemma (Theorem \ref{thm.aubin1} if $s\ge 1/2$ or Theorem \ref{thm.aubin2}
if $s<1/2$) to obtain the existence of
a subsequence which is not relabeled such that, as $(\eps,\tau)\to 0$,
$$
  u^{(\tau)}\to u \quad\mbox{strongly in }L^\gamma(Q_T) \quad\mbox{for all } 
	1 \le \gamma < p(s).
$$
In particular, $u^{(\tau)}\to u$ a.e.\ in $Q_T$.
By estimates \eqref{4.aubin1}, \eqref{4.eps}, and \eqref{4.aubin2}, we have, 
up to subsequences,
\begin{align*}
  (u_i^{(\tau)})^s \rightharpoonup u_i^s &\quad\mbox{weakly in }
	L^2(0,T;H^1(\Omega)), \\
  \eps w^{(\tau)} \to 0 &\quad\mbox{strongly in }L^2(0,T;H^m(\Omega)), \\
	\tau^{-1}(u^{(\tau)}-\sigma_\tau u^{(\tau)}) \rightharpoonup \pa_t u
	&\quad\mbox{weakly in }L^{q'}(0,T;W^{m,q}(\Omega)').
\end{align*}

We have shown in the proof of Lemma \ref{lem.ut} that 
$(u_i^{(\tau)}p_i(u^{(\tau)}))$ is bounded in $L^{p(s)/(s+1)}(Q_T)$.
Taking into account the a.e.\ convergence
$u_i^{(\tau)}p_i(u^{(\tau)})\to u_ip_i(u)$ in $Q_T$, we infer that
$$
  u_i^{(\tau)}p_i(u^{(\tau)}) \to u_ip_i(u) \quad\mbox{strongly in }L^1(Q_T).
$$
Furthermore, we proved that 
$(\eps^{\eta/r(s)+(\eta+1)/2}A_{ij}^0(u^{(\tau)})\na u_j^{(\tau)})$ 
is bounded in $L^{q_2}(Q_T)$ with $q_2=2r(s)/(r(s)+2)$ such that
\begin{align*}
  \eps A_{ij}^0(u^{(\tau)})\na u_j^{(\tau)}
	&= \eps^{1-\eta/r(s)-(\eta+1)/2}\cdot\eps^{\eta/r(s)+(\eta+1)/2}
	A_{ij}^0(u^{(\tau)})\na u_j^{(\tau)} \\
	&\to 0 \quad\mbox{strongly in }L^1(Q_T).
\end{align*}
Here, we used the fact that $\eta/r(s)+(\eta+1)/2<1$ such that 
$\eps^{1-\eta/r(s)-(\eta+1)/2}\to 0$ as $\eps\to 0$.
We know from \eqref{4.Lp} that $(\eps^{\eta/r(s)}u_i^{(\tau)})$ is bounded in
$L^2(Q_T)$. Consequently,
$$
  \eps^\eta(u_i^{(\tau)})^2 
	= \eps^{\eta(1-2/r(s))}\big(\eps^{\eta/r(s)}u_i^{(\tau)}\big)^2 \to 0
	\quad\mbox{strongly in }L^1(Q_T),
$$
since $\eps^{\eta(1-2/r(s))}\to 0$ as $\eps\to 0$ because of $r(s)>2$.
Finally, $f_i(u^{(\tau)})\to f_i(u)$ a.e.\ and the uniform bound 
$\|f_i(u^{(\tau)}\|_{L^{q_4}(Q_T)}\le C$ with $q_4=p(s)/(1+\sigma)>1$ imply that
$$
  f_i(u^{(\tau)})\to f_i(u) \quad\mbox{strongly in }L^1(Q_T).
$$
Then, performing the limit $(\eps,\tau)\to 0$ in \eqref{4.aux3} with
$\phi\in L^\infty(0,T;W^{m,\infty}_\nu(\Omega))$ , it follows
that $u$ solves \eqref{1.weak2} for such test functions.
A density argument shows that, in fact, $u$ solves \eqref{1.weak2} for
$\phi\in L^q(0,T;W^{m,q}_\nu(\Omega))$, finishing the proof.

\begin{remark}[Weak formulation]\rm\label{rem.weak}
In the superlinear case $s>1$, the solution constructed in the previous
proof satisfies \eqref{1.eq} even in the weak sense \eqref{1.weak1} with test 
functions $\phi\in L^q(0,T;W^{1,q}(\Omega))$. In order
to see this, it is sufficient to show that
$$
  A_{ij}(u^{(\tau)})\na u^{(\tau)}_j\rightharpoonup A_{ij}(u)\na u_j
	\quad\mbox{weakly in }L^{q'}(Q_T)
$$
for some $1<q'\le 2$. Because of the structure of $A_{ij}$, we only need
to verify that
\begin{align*}
  u^{(\tau)}_i(u^{(\tau)}_j)^{s-1}\na u_j^{(\tau)} 
	\rightharpoonup u_iu_j^{s-1}\na u_j
	&\quad\mbox{weakly in }L^{q'}(Q_T), \\
  (u^{(\tau)}_i)^s\na u_j^{(\tau)} 
	\rightharpoonup u_i^s\na u_j
	&\quad\mbox{weakly in }L^{q'}(Q_T).
\end{align*}
Indeed, we have the convergences $u_i^{(\tau)}\to u_i$ strongly in
$L^\gamma(Q_T)$ for any $2<\gamma<p(s)$ and $(u_i^{(\tau)})^s
\rightharpoonup u_i^s$ weakly in $L^2(0,T;H^1(\Omega))$ and hence,
$$
  u^{(\tau)}_i(u^{(\tau)}_j)^{s-1}\na u_j^{(\tau)} 
	= s^{-1}u^{(\tau)}_i\na(u^{(\tau)}_j)^s
	\rightharpoonup s^{-1}u_i\na u_j^s = u_iu_j^{s-1}\na u_j
	\quad\mbox{weakly in }L^{q'}(Q_T),
$$
choosing $q'=2\gamma/(\gamma+2)>1$. For the remaining convergence, we need to
integrate by parts. It holds for $\phi_i\in L^q(0,T;W^{2,q}_\nu(\Omega))$ that
\begin{align*}
  \int_0^T & \int_\Omega(u_i^{(\tau)})^s\na u_j^{(\tau)}\cdot\na\phi_i dxdt \\
	&= -\int_0^T\int_\Omega u_j^{(\tau)}\na(u_i^{(\tau)})^s\cdot\na\phi_i dxdt
	- \int_0^T\int_\Omega(u_i^{(\tau)})^s u_j^{(\tau)}\Delta\phi_i dxdt \\
	&\phantom{xx}{}\to -\int_0^Tu_j\na u_i^s\cdot\na\phi_i dxdt 
	- \int_0^T\int_\Omega u_i^s u_j\Delta\phi_i dxdt
	= \int_0^t\int_\Omega u_i^s\na u_j\cdot\na\phi_i dxdt.
\end{align*}
A density argument shows that the weak formulation also holds for
$\phi_i\in L^q(0,T;W^{1,q}(\Omega))$.
\qed
\end{remark}

\begin{remark}[Vanishing self-diffusion]\label{rem.aii}\rm
Assume that $a_{i0}>0$ and $a_{ii}=0$. 
The difficulty is that we obtain a uniform bound only for $\na (u_i^{(\tau)})^{s/2}$
instead for $\na(u_i^{(\tau)})^s$ in $L^2(Q_T)$. In order to compensate this
loss of regularity, we need additional assumptions, namely either 
$s>\max\{1,d/2\}$ (superlinear rates); or $0<s<1$, $d=1$, and
$\sigma<s+1$ (sublinear rates).
Under these conditions, the statement of Theorem \ref{thm.ex2} holds true.

For the proof, we remark that
the regularity for $(u_i^{(\tau)})^s$ in $L^2(0,T;H^1(\Omega))$ is employed
in the estimate of $u_i^{(\tau)}$ in $L^{p(s)}(Q_T)$. If only $(u_i^{(\tau)})^{s/2}$
is bounded in $L^2(0,T;H^1(\Omega))$, the Gagliardo-Nirenberg inequality
gives a weaker result: for $0<s<1$ with $\theta=ds/(ds+2)$ and $\rho=s+2/d$,
\begin{align*}
  \|u_i^{(\tau)}\|_{L^{\rho}(Q_T)}^\rho
	&= \|(u_i^{(\tau)})^{s/2}\|_{L^{2\rho/s}(Q_T)}^{2\rho/s}
	\le C\int_0^T\|(u_i^{(\tau)})^{s/2}\|_{H^1(\Omega)}^{2\theta\rho/s}
	\|(u_i^{(\tau)})^{s/2}\|_{L^{2/s}(\Omega)}^{2(1-\theta)\rho/s}dt \\
	&\le C\|(u_i^{(\tau)})^{s/2}\|_{L^2(0,T;H^1(\Omega))}^2
	\|u_i^{(\tau)}\|_{L^\infty(0,T;L^1(\Omega))}^{(1-\theta)\rho} \le C,
\end{align*}
since $2\theta\rho/s=2$; and for $s>1$ with
$\theta=d/(d+2)$ and $\rho=s+2s/d$,
$$
  \|u_i^{(\tau)}\|_{L^\rho(Q_T)}^\rho 
	\le C\|(u_i^{(\tau)})^{s/2}\|_{L^2(0,T;H^1(\Omega))}^2
	\|u_i^{(\tau)}\|_{L^\infty(0,T;L^s(\Omega))}^{(1-\theta)\rho} \le C,
$$
since $2\theta\rho/s=2$. Consequently, $(u_i^{(\tau)})$
is bounded in $L^\rho(Q_T)$ with $\rho=s+(2/d)\max\{1,s\}$. 

We claim that this estimate is sufficient to derive a bound for the discrete
time derivative. Since the $\eps$-terms in \eqref{4.time} do not need
the estimate for $u_i^{(\tau)}$ in $L^\rho(Q_T)$, it is sufficient to bound
$u_i^{(\tau)}p_i(u^{(\tau)})$ and $(u_i^{(\tau)})^{\sigma+1}$ in some $L^{q'}(Q_T)$
with $q'>1$. This is possible as long as $\rho>s+1$ and $\rho>\sigma+1$, respectively.
If $0<s<1$, these two inequalities are equivalent to $d=1$ and $\sigma<s-1+d/2=s+1$.
If $s>1$ (in this case $\sigma=1$), they give the restriction $s>d/2$, thus 
$s>\max\{1,d/2\}$. This shows the claim.
\qed
\end{remark}


\section{Additional and auxiliary results}\label{sec.aux}

\subsection{Detailed balance condition}\label{sec.db}

We wish to interpret the detailed balance condition \eqref{1.db}
and to explain how the numbers $\pi_i$ can be computed from the coefficients
$(a_{ij})$. We assume that the coefficients are normalized in the sense
that $a_{ij}\ge 0$ and $\sum_{k=1,\,k\neq j}a_{kj}\le 1$ for all $i,j$.
The idea is to use a probabilistic approach, interpreting
the coefficients $a_{ij}$ as the transition rates between two
discrete states $i$ and $j$ of the state space $S:=\{1,\ldots,n\}$. Then
$$
  a_{ij} = \mbox{P}(X_k=j|X_{k-1}=i)
$$
is the conditional probability for a random variable $X:\N\to S$.
This variable represents the Markov chain associated to the stochastic matrix
$Q=(Q_{ij})_{i,j}\in\R^{n\times n}$, defined by
$Q_{ij}=a_{ij}$ for $i\neq j$ and $Q_{ii}=1-\sum_{i=1,\,i\neq j}a_{ij}$ for
$i=1,\ldots,n$. 
A Markov chain is called reversible if there exists a probability distribution
$\pi=(\pi_1,\ldots,$ $\pi_n)$ on $S$ (called an invariant measure) such that
\begin{equation}\label{2.db}
  \pi_i a_{ij} = \pi_j a_{ji}, \quad i,j=1,\ldots,n.
\end{equation}
The Markov chain can be interpreted as a directed graph, where the states $i\in S$
are the nodes and the edges are labeled by the probabilities $a_{ij}$ 
going from state $i$ to state $j$. 

The state space $S$ can be partitioned into so-called communicating classes.
We write $i\to j$ if there exist $i_0,i_1,\ldots,i_{n+1}\in S$ such that
$a_{i_0,i_1}a_{i_1,i_2}\cdots a_{i_n,i_{n+1}}>0$ for $i_0=i$ and $i_{n+1}=j$.
We say that $i$ communicates with $j$ if both $i\to j$ and $j\to i$. A set
of states $\sigma\subset S$ is a communicating class if every pair in 
$\sigma$ communicates with each other. This defines an equivalence relation, 
and communicating classes are the equivalence classes.

Consider the following properties:
\begin{description}
\item[(A1)] For all $i,j\in S$, it holds that either $a_{ij}=a_{ji}=0$
or $a_{ij}a_{ji}>0$.
\item[(A2)] For any periodic cycle $i_0,i_1,\ldots,i_{m+1}=i_0$, 
$$
  \prod_{k=0}^m a_{i_k,i_{k+1}} = \prod_{k=0}^m a_{i_{k+1},i_k}.
$$
\end{description}
The detailed balance condition \eqref{2.db} implies (A1) and (A2). It is shown in 
\cite{Suo79} that the converse is true and that the invariant measure $\pi$ can be
constructed explicitly.

\begin{proposition}\label{prop.im}
Let (A1)-(A2) hold. 
Then there exists an invariant measure $\pi=(\pi_1,\ldots,$ $\pi_n)$ such that
the detailed balance condition \eqref{2.db} is satisfied. Moreover,
$\pi$ can be computed explicitly by choosing an $i_0$ in each communicating
class and defining $\pi_j$ for $i_0$ and $j$ belonging in the same class by
$$
  \pi_j := \prod_{k=1}^{n-1}\frac{a_{i_k,i_{k+1}}}{a_{i_{k+1},i_k}}
$$
depending only on $i_0$ and $j$, where $i_1,i_2,\ldots,i_n=j$ are such that
$a_{i_k,i_{k+1}}>0$ for $k=0,\ldots,n-1$.
\end{proposition}

For instance, if $n=3$, we need to suppose (according to (A2)) that
\begin{equation}\label{2.n3}
  a_{12}a_{23}a_{31} = a_{13}a_{32}a_{21},
\end{equation}
and the invariant measure is given by $\pi=c(1,a_{12}/a_{21},a_{13}/a_{31})$,
where $c=(1+a_{12}/a_{21}+a_{13}/a_{31})^{-1}$.

The following result relates the detailed balance condition
and the symmetry of the matrix $H(u)A(u)$.

\begin{proposition}\label{prop.db}
The following three properties are equivalent:
\begin{description}
\item[(i)] Graph-theoretical condition: (A1) and (A2) hold.
\item[(ii)] Detailed balance condition: $\pi_i a_{ij} = \pi_j a_{ji}$ for
$i\neq j$.
\item[(iii)] Symmetry: The matrix $H(u)A(u)$ is symmetric.
\end{description}
\end{proposition}

\begin{proof}
The implication (i) $\Leftarrow$ (ii) is shown in Proposition \ref{prop.im}.
The converse can be proved directly using the detailed balance condition.
Finally, the equivalence (ii) $\Leftrightarrow$ (iii) follows from an explicit
calculation of $H(u)A(u)$.
\end{proof}

\begin{remark}\rm
The equivalence of the symmetry of $H(u)A(u)$ and the detailed balance condition
is related to the Onsager principle of thermodynamics. Indeed,
the diffusion matrix $B=A(u)H(u)^{-1}$ in
$$
  \pa_t u - \diver(B\na w) = f(u),
$$
where $w=h'(u)$ is the vector of entropy variables, is the Onsager matrix which
is symmetric, according to Onsager, if and only if the thermodynamic system is
time-reversible. Time-reversibility means that the Markov chain associated to
the matrix $(a_{ij})$ is reversible, and the symmetry of $B$ is equivalent to
the symmetry of $H(u)A(u)$. Thus, the equivalence (ii) $\Leftrightarrow$ (iii)
corresponds to the equivalence of the symmetry of $B$ and the time-reversibility.
For details on the detailed balance principle in thermodynamics, we refer to
\cite{DeMa62}.
\qed
\end{remark}


\subsection{Nonlinear Aubin-Lions lemmas}

Let $\Omega\subset\R^d$ ($d\ge 1$) be a bounded domain with Lipschitz
boundary. Let $(u^{(\tau)})$ be a family of nonnegative functions 
which are piecewise constant in time with uniform time step size $\tau>0$.
We introduce the time shift operator
$(\sigma_\tau u^{(\tau)})(t)=u^{(\tau)}(t-\tau)$ for $t\ge\tau$.

If there exist uniform estimates for the gradient $(\na u^{(\tau)})$ and the
discrete time derivative $\tau^{-1}(u^{(\tau)}-\sigma_\tau u^{(\tau)})$, 
then, by the Aubin-Lions theorem and under suitable conditions on the spaces,
$(u^{(\tau)})$ is relatively compact in some $L^q$ space. 
In the case of nonlinear transition
rates, we obtain uniform estimates only for $(\na(u^{(\tau)})^s)$, where $s>0$.
Then relative compactness follows from a nonlinear version of the Aubin-Lions
theorem \cite{CJL13}. We recall a special case of this result.

\begin{theorem}[Nonlinear Aubin-Lions lemma for $s\ge 1/2$]\label{thm.aubin1} 
Let $s\ge 1/2$, $m\ge 0$, $1\le q<\infty$,
and there exists $C>0$ such that for all $\tau>0$,
$$
  \|(u^{(\tau)})^s\|_{L^2(0,T;W^{1,q}(\Omega))}
	+ \tau^{-1}\|u^{(\tau)}-\sigma_\tau u^{(\tau)}\|_{L^1(\tau,T;H^m(\Omega)')}
	\le C.
$$ 
Then there exists a subsequence of $(u^{(\tau)})$, which is not relabeled,
such that, as $\tau\to 0$,
$$
  u^{(\tau)}\to u\quad\mbox{strongly in }L^{2s}(0,T;L^{ps}(\Omega)),
$$
where $p\ge \max\{1,1/s\}$ is such that the embedding $W^{1,q}(\Omega)
\hookrightarrow L^p(\Omega)$ is compact.
\end{theorem}

Theorem \ref{thm.aubin1} can be extended to the case $s<1/2$ if 
$(u^{(\tau)})$ is additionally bounded in $L^\infty(0,T;L^1(\Omega))$ which
generally follows from the entropy inequality. This result is new.

\begin{theorem}[Nonlinear Aubin-Lions lemma for $s < 1/2$]\label{thm.aubin2}
Let $\max\{0,1/2-1/d\}<s<1/2$, $m\ge 0$, 
and there exists $C>0$ such that for all $\tau>0$,
$$
  \|u^{(\tau)}\|_{L^\infty(0,T;L^1(\Omega))}
  + \|(u^{(\tau)})^s\|_{L^2(0,T;H^1(\Omega))}
	+ \tau^{-1}\|u^{(\tau)}-\sigma_\tau u^{(\tau)}\|_{L^1(\tau,T;H^m(\Omega)')}
	\le C.
$$
Then there exists a subsequence of $(u^{(\tau)})$, which is not relabeled,
such that, as $\tau\to 0$,
$$
  u^{(\tau)}\to u\quad\mbox{strongly in }L^{1}(0,T;L^{1}(\Omega)).
$$
\end{theorem}

\begin{proof}
The result follows from Theorem \ref{thm.aubin1} 
and the H\"older inequality. Indeed, we have
\begin{align*}
  \|\na (u^{(\tau)})^{1/2} & \|_{L^2(0,T;L^{1/(1-s)}(\Omega))}
	= (2s)^{-1}\|(u^{(\tau)})^{1/2-s}\na(u^{(\tau)})^s
	\|_{L^2(0,T;L^{1/(1-s)}(\Omega))} \\
	&\le (2s)^{-1}\|(u^{(\tau)})^{(1-2s)/2}\|_{L^\infty(0,T;L^{2/(1-2s)}(\Omega))}
	\|\na(u^{(\tau)})^s\|_{L^2(0,T;L^2(\Omega))} \\
	&= \|u^{(\tau)}\|_{L^\infty(0,T;L^1(\Omega))}^{(1-2s)/2}
	\|\na(u^{(\tau)})^s\|_{L^2(0,T;L^2(\Omega))} \le C.
\end{align*}
Therefore, $(u^{(\tau)})^{1/2}$ is uniformly bounded in
$L^2(0,T;W^{1,1/(1-s)}(\Omega))$. By Rellich-Kondra\-chov's theorem,
the embedding $W^{1,1/(1-s)}(\Omega)\hookrightarrow L^2(\Omega)$ is compact
for $s>0$ if $d\le 2$ and $s>1/2-1/d$ if $d\ge 3$. 
Applying Theorem \ref{thm.aubin1} with $s=1/2$, $q=1/(1-s)$, and $p=2$, we infer that 
$(u^{(\tau)})$ is relatively compact in $L^1(0,T;L^1(\Omega))$.
\end{proof}


\subsection{Increasing entropies}\label{sec.incr}

If detailed balance or a weak cross-diffusion condition hold, 
we have shown that the entropy is nonincreasing in time
along solutions to \eqref{1.eq}-\eqref{1.bic}. In this section, we show
that the entropy may be increasing for small times if these conditions do
not hold. To simplify the presentation, we restrict ourselves to the case 
$n=3$ (three species), $s=1$ (linear transition rates), and $\Omega=(0,1)$.

\begin{lemma}[Vanishing diffusion coefficients $a_{i0}$]
Let $a_{13}=a_{32}=a_{21}=1$ and $a_{ij}=0$ else.
For any $\eps>0$, there exist initial data $u^0$ such that
$$
  \frac{d\HH}{dt}[u^0] \ge \frac{1}{\eps}.
$$
In particular, if $t\mapsto \HH[u(t)]$ is continuous, there exists $t_0>0$
such that $t\mapsto \HH[u(t)]$ is increasing on $[0,t_0]$.
\end{lemma}

\begin{proof}
Observe that \eqref{2.n3} is not satisfied, and hence detailed balance does not hold.
Furthermore, we have
$$
  H(u)A(u) = \begin{pmatrix} 
	1/u_1 & 0 & 0 \\ 0 & 1/u_2 & 0 \\ 0 & 0 & 1/u_3 \end{pmatrix}
	\begin{pmatrix}
	u_3 & 0 & u_1 \\ u_2 & u_1 & 0 \\ 0 & u_3 & u_2 \end{pmatrix}
	= \begin{pmatrix}
	u_3/u_1 & 0 & 1 \\ 1 & u_1/u_2 & 0 \\ 0 & 1 & u_2/u_3 \end{pmatrix}.
$$
Let $0<\eps<0.5$ and define $u^0=(u^0_1,u^0_2,u^0_3)$ by 
$u_1^0(x)=1$ for $x\in(0,1)$ and
\begin{align*}
  u_2^0(x) &= \left\{\begin{array}{ll}
	3 &\mbox{for } 0<x<0.5, \\
	3 - \eps^{-1}(x-0.5) &\mbox{for } 0.5<x<0.5+\eps, \\
	2 &\mbox{for }0.5+\eps<x<1,
	\end{array}\right. \\
  u_3^0(x) &= \left\{\begin{array}{ll}
	9 &\mbox{for } 0<x<0.5, \\
	9 + \eps^{-1}(x-0.5) &\mbox{for } 0.5<x<0.5+\eps, \\
	10 &\mbox{for }0.5+\eps<x<1,
	\end{array}\right. 
\end{align*}
Then
\begin{align*}
  \int_0^1(\pa_xu^0)^\top H(u^0)A(u^0)\pa_x u^0 dx
	&= \frac{1}{\eps^2}\int_{0.5}^{0.5+\eps}\bigg(\frac{1}{u_2^0(x)} - 1
	+ \frac{u_2^0(x)}{u_3^0(x)}\bigg)dx \\
	&\le \frac{1}{\eps}\bigg(\frac12 - 1 + \frac{3}{9}\bigg) = -\frac{1}{6\eps},
\end{align*}
which implies that $(d\HH/dt)[u^0]\ge 1/(6\eps)$.
\end{proof}

One may ask if a similar result as above holds if the diffusion coefficients
$a_{i0}$ do not vanish, since they give positive contributions to the
entropy production. The next lemma shows that the entropy may be increasing
even if $a_{i0}>0$ is chosen arbitrarily.

\begin{lemma}[Positive diffusion coefficients $a_{i0}$]
Let $a_{13}=a_{32}=a_{21}=1$, $a_{i0}>0$ for $i=1,2,3$, and $a_{ij}=0$ else.
For any $\eps>0$, there exist initial data $u^0$ such that
$$
  \frac{d\HH}{dt}[u^0] \ge \frac{1}{\eps}.
$$
In particular, if $t\mapsto \HH[u(t)]$ is continuous, there exists $t_0>0$
such that $t\mapsto \HH[u(t)]$ is increasing on $[0,t_0]$.
\end{lemma}

\begin{proof}
We choose the initial datum
\begin{align*}
  u_1^0(x) &= \frac{a_{20}(2a_{20}+a_{30})}{8a_{20}+4a_{30}}, \\
	u_2^0(x) &= \left\{\begin{array}{ll}
	4a_{20} &\mbox{for } 0<x<0.5, \\
	a_{20}(4 - \eps^{-1}(x-0.5)) &\mbox{for } 0.5<x<0.5+\eps, \\
	3a_{20} &\mbox{for }0.5+\eps<x<1,
	\end{array}\right. \\
	u_3^0(x) &= \left\{\begin{array}{ll}
	8a_{20}+4a_{30} &\mbox{for } 0<x<0.5, \\
	a_{20}(8 - \eps^{-1}(x-0.5)) + 4a_{30} &\mbox{for } 0.5<x<0.5+\eps, \\
	9a_{20} + 4a_{30} &\mbox{for }0.5+\eps<x<1,
	\end{array}\right.
\end{align*}
Then
\begin{align*}
  \int_0^1 & (\pa_xu^0)^\top H(u^0)A(u^0)\pa_x u^0 dx \\
	&= \int_{0.5}^{0.5+\eps}\bigg(\frac{u_1^0}{u_2^0}(\pa_x u_2^0)^2
	+ \frac{a_{20}}{u_2^0}(\pa_x u_2^0)^2 + \pa_x u_2^0\pa_x u_3^0
	+ \frac{u_2^0+a_{30}}{u_3^0}(\pa_x u_3^0)^2\bigg)dx \\
	&\le \frac{a_{20}^2}{\eps^2}
	\int_{0.5}^{0.5+\eps}\bigg(\frac{2a_{20}+a_{30}}{3(8a_{20}+4a_{30})}
	\frac{a_{20}}{3a_{20}} - 1 + \frac{4a_{20} + a_{30}}{8a_{20}+4a_{30}}\bigg)dx \\
	&\le \frac{a_{20}^2}{\eps^2}\bigg(\frac{1}{12} - \frac{1}{3} - 1 + \frac12\bigg)
	= -\frac{a_{20}^2}{12\eps},
\end{align*}
which proves the result.
\end{proof}


\end{document}